\documentclass[a4paper,11pt]{article}

\usepackage[textwidth=125mm, textheight=195mm]{geometry}
\usepackage[english]{babel}
\usepackage{graphicx}
\usepackage{amsmath}
\usepackage{amsfonts}
\usepackage{amsthm}
\usepackage{float}
\usepackage{epigraph}

\begin{document}

\newcommand{\wk}{\mbox{$\,<$\hspace{-5pt}\footnotesize )$\,$}}

\numberwithin{equation}{section}
\newtheorem{teo}{Theorem}
\newtheorem{lemma}{Lemma}

\newtheorem{coro}{Corollary}
\newtheorem{prop}{Proposition}

\newtheorem{definition}{Definition}
\theoremstyle{remark}
\newtheorem{remark}{Remark}

\newtheorem{scho}{Scholium}
\newtheorem{open}{Question}
\newtheorem{example}{Example}
\numberwithin{example}{section}
\numberwithin{lemma}{section}
\numberwithin{prop}{section}
\numberwithin{teo}{section}
\numberwithin{definition}{section}
\numberwithin{coro}{section}
\numberwithin{figure}{section}
\numberwithin{remark}{section}
\numberwithin{scho}{section}

\bibliographystyle{abbrv}

\title{Concepts of curvatures in normed planes}
\date{}
\author{Vitor Balestro\footnote{Corresponding author} \\ CEFET/RJ Campus Nova Friburgo \\ 28635000 Nova Friburgo - RJ \\ Brazil \\ vitorbalestro@id.uff.br \and Horst Martini \\ Fakult\"{a}t f\"{u}r Mathematik \\ Technische Universit\"{a}t Chemnitz \\ 09107 Chemnitz\\ Germany \\ martini@mathematik.tu-chemnitz.de \and  Emad Shonoda \\ Department of Mathematics \& Computer Science \\ Faculty of Science \\ Port Said Unversity \\42521 Port Said\\ Egypt \\ en$\_$shonoda@yahoo.de}

\maketitle

\epigraph{\emph{``Everyone knows what a curve is, until he has studied enough
Mathematics to become confused through the countless number
of possible exceptions."}}{Felix Klein}

\begin{abstract}
The theory of classical types of curves in normed planes is not strongly developed. In particular, the knowledge on existing concepts of curvatures of planar curves is widespread and not systematized in the literature. Giving a comprehensive overview on geometric properties of and relations between all introduced curvature concepts, we try to fill this gap. Certainly, this yields a basis for further research and also for possible extensions of the whole existing framework. In addition, we derive various new results referring in full broadness to the variety of known curvature types in normed planes. These new results involve characterizations of curves of constant curvature, new characterizations of Radon planes and the Euclidean subcase, and analogues to classical statements like the four vertex theorem and the fundamental theorem on planar curves. We also introduce a new curvature type, for which we  verify corresponding properties. As applications of the little theory developed in our expository paper, we study the curvature behaviour of curves of constant width and obtain also new results on notions like evolutes, involutes, and parallel curves.
\end{abstract}



\tableofcontents

\section{Introduction}

The generalization of notions from Euclidean geometry to non-Euclidean geometries having the Euclidean one as subcase very often yields the following situation: one fixed Euclidean notion has a multiplicity of extended analogues, since properties which naturally coincide in the initial (Euclidean) situation can occur in different generalized forms. Regarding Minkowski geometry (i.e., the geometry of finite dimensional real Banach spaces), a well known example is given by the large variety of orthogonality concepts in normed (= Minkowski) planes; another one is presented by the curvature types introduced (up to now) for the study of classical curves in normed planes. In the first case comprehensive surveys exist (see, e.g., \cite{alonso}), in the latter case these concepts are only widespread and partially hidden in the literature, and a systematic representation of their variety and relations to each other (as good starting point for further research) is still missing. We refer to the paper \cite{martiniandwu}, where this unsatisfying situation regarding curve theory in normed planes is already described, and another related and comprehensive source (even referring to non-symmetric metrics, called gauges) is \cite{Ja-Ma-Ri}. It is our goal to fill the gap outlined here for the particular viewpoint of curvature notions. In the present expository paper we want to present geometric properties of the different existing curvature types in normed planes, to shed light on the relations between them and even to develop, based on this, a little theory in modern terms. We also derive various new results related to these notions. More precisely, we study, compare and classify the notions of circular, normal, and Minkowski curvature of curves embedded into normed planes. These concepts of curvatures in normed planes were introduced and treated by authors such as Biberstein \cite{biberstein}, Busemann \cite{Bus3}, Ghandehari \cite{Ghan1,Ghan2}, and Petty \cite{Pet}. Referring to these extensions of the Euclidean curvature, we characterize curves of constant curvatures, present characterizations of Radon planes or of the Euclidean subcase, and establish analogues of classical statements like the four vertex theorem and the fundamental theorem on planar curves. Completing the whole picture, we introduce a new type of curvature via arc length. Also for this curvature concept we prove results of the same kind as mentioned above. It is also our aim to unify and modernize the methods that can be used in this little field. As applications of the developed machinery, we get also results on curvature properties of curves of constant width, and we also study notions like evolutes, involutes, and parallel curves from the viewpoint of singularity theory, characterizing the involved concepts as singular points of certain maps.\\

In the mentioned papers \cite{Bus3} and \cite{Pet}, Busemann and Petty use an auxiliary Euclidean structure to deal with the so-called \emph{isoperimetrices}, which are solutions of the isoperimetric problem. Such solutions present homothety classes, and their representatives are said to be \emph{normalized} if their perimeters equal twice their area (see also \cite[Section 4.4]{thompson}). We follow the approach given in \cite{martiniantinorms} where, in the planar setting, the new norm is induced by the isoperimetrix taken as unit circle. \\

Roughly speaking, Biberstein's approach to define curvature is based on measuring the variation (in the sense of rotation) of the tangent vector with respect to the arc length of the curve. He re-obtains Petty's \emph{Minkowski curvature} (see again \cite{biberstein} and \cite{Pet}). Similarly, Biberstein defines his \emph{anti-curvature} which, roughly speaking, describes the variation of the rotation of the \emph{normal} vector with respect to the arc length (``normal" has not a unique meaning in Minkowski geometry, and we shall explain later what we mean here). Taylor \cite[Sections 3.3 and 3.4]{taylor} also studied these concepts. It turns out that, as we shall see, Biberstein's anti-curvature is Petty's \emph{isoperimetric curvature}. Using this machinery, both authors define (in the same way) analogues to Euclidean Frenet formulas. The same curvature types and formulas were also obtained in \cite{haddou}. The new curvature type, introduced in this paper via arc length and already mentioned, turns out to be precisely the \emph{Minkowski curvature in the anti-norm}. Also, Gage \cite{gage1} noticed that a certain curvature motion in a relative geometry can be regarded as a motion by the isoperimetric curvature of a Minkowski geometry (see also \cite{gage2}). \\

Petty \cite{Pet} defines a further concept that he called \emph{circular curvature}. As we will see, this can be regarded as the inverse of the radius of the \emph{osculating circle} (which is a 2nd-order contact Minkowski circle attached to the curve at the considered point). This curvature type was also used by Craizer \cite{craizer}, Craizer et al. \cite {Cra-Tei-Ba} and Ghandehari \cite{Ghan1,Ghan2}. Guggenheimer \cite{Gug1} works with an osculating anti-circle, obtaining, as we shall see, the isoperimetric curvature. A related discrete framework appears in \cite{craizermartini} and \cite{Tei-Cra-Ba}.

\section{Curvature in the Euclidean plane}\label{curvatureeuclid}

We want to describe the curvature concepts in normed planes in a way to emphasize their geometric meaning. For this sake, we shall start by taking a look at the usual curvature of curves in the Euclidean plane. Intuitively, by curvature the rotation of the tangent field with respect to the traveled length along the curve is measured (see Figure \ref{bigsmallcurvature}).

\begin{figure}[h]
\centering
\includegraphics{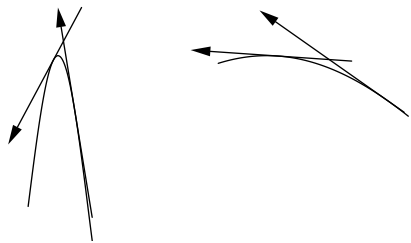}
\caption{``Large" and ``small" curvatures.}
\label{bigsmallcurvature}
\end{figure}

In the Euclidean plane (with usual fixed orientation), we may parametrize the studied curve by arc length, and the ``amount of rotation" can be measured considering the area swept by the tangent field when identified within the unit circle (Figure \ref{euclidcurvature}). \\

\begin{figure}[h]
\centering
\includegraphics{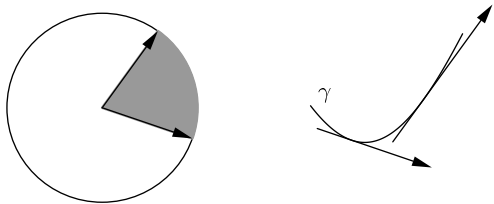}
\caption{The swept area by the curve's tangent field.}
\label{euclidcurvature}
\end{figure}

Formally, we define the curvature as follows: let $\gamma:[0,l]\rightarrow\mathbb{R}^2$ be a smooth curve parametrized by the \emph{arc length} $s$, and let $u:[0,l]\rightarrow\mathbb{R}$ be the function which associates to each $s \in [0,l]$ twice the (signed) value of the area of the sector between $\gamma'(0)$ and $\gamma'(s)$, where the symbol $'$ always denotes the derivative with respect to $s$. Then we define the (signed) curvature $k(s)$ of $\gamma$ at $\gamma(s)$ to be
\begin{align*} k(s) := u'(s), \ \ s \in [0,l].
\end{align*}
This definition can be rewritten in other terms. In \cite{manfredo}, for example, the curvature is regarded as the (signed) length of the derivative of the (unit) tangent field. Since $\gamma'(s)$ is a point of the unit circle for each $s \in [0,l]$, it follows that $\gamma''(s)$ points in the direction normal to $\gamma'(s)$. Letting $n(s)$ be the positively oriented normal vector to $\gamma'(s)$ of unit length, we write
\begin{align*} \gamma''(s) = k(s)n(s), \ \ s \in [0,l].
\end{align*}
And to verify that both definitions coincide, one just has to let $\varphi:[0,2\pi]\rightarrow\mathbb{R}^2$ be a (positively oriented) parametrization of the unit circle by twice the sector area (which is also an arc-length parametrization). Considering $u(s)$ as above, writing $\varphi(u(s))=\gamma'(s)$ and differentiating, we have
\begin{align*} \gamma''(s) = u'(s)\frac{d\varphi}{du} = u'(s)n(s),
\end{align*}
as we wished. Here one should notice that the derivative of the normal vector field $n(s)$ can be written as $n'(s) = -k(s)\gamma'(s)$. This can be intuitively understood by remembering the fact that the area determined in the unit circle by the unit normal field $n(s)$ equals the area determined by the unit tangent field $\gamma'(s)$. As we will see later, the differentiations of the tangent and the normal field do not necessarily yield the same curvature in an arbitrary normed plane. \\

There is a third way to obtain curvature in the Euclidean plane, which will later give birth to a certain curvature type in the normed plane (namely, the \emph{circular curvature}). Actually, we obtain the \emph{curvature radius} of a curve in each point. Let $\gamma:[0,l]\rightarrow \mathbb{R}^2$ be a smooth curve parametrized by the arc length $s$, and let $\sigma:[0,2\pi]\rightarrow\mathbb{R}^2$ be an arc-length (or twice the sector area) parametrization of the unit circle (we call this parameter $u$). Assume that $\gamma$ does not contain segments. Then we may locally consider a re-parametrization $u = s(u)$ of $\gamma$ in such a way that the tangent vector to $\gamma$ at $\gamma(s(u))$ points in the direction (and orientation) of $\sigma'(u)$. We write 
\begin{align*} \frac{d}{du}\gamma(s(u)) = \rho(u)\sigma'(u).
\end{align*}
Geometrically, we are fitting a circle (with radius $\rho(u)$), attached at the point $\gamma(s(u))$ of the curve, in such a way that the tangent vectors of the curve and of the circle coincide at this point. The number $\rho(u)$ is called the \emph{curvature radius} of $\gamma$ at $s(u)$, and we claim that
\begin{align*} k(s(u)) = \frac{1}{\rho(s(u))},
\end{align*}
where $k(s(u))$ is, as usual, the curvature of $\gamma$ at $s(u)$. Indeed,
\begin{align*} \rho(u)\sigma'(u) = \frac{d}{du}\gamma(s(u)) = \frac{ds}{du}(u)\frac{d\gamma}{ds}(s(u)),
\end{align*}
and since both $\sigma'(u)$ and $\frac{d\gamma}{ds}(s(u))$ are unit, we have
\begin{align*} \frac{1}{\rho(u)} = \frac{du}{ds}(s(u)).
\end{align*}
The latter is, by definition, the curvature of $\gamma$ at $s(u)$. The circle of radius $\rho(u)$ and passing through $\gamma(s(u))$ in the described way is called the \emph{osculating circle} (or \emph{circle of curvature}) of $\gamma$ at $\gamma(s(u))$. Figure \ref{osculating} illustrates this construction. \\

\begin{figure}[h]
\centering
\includegraphics{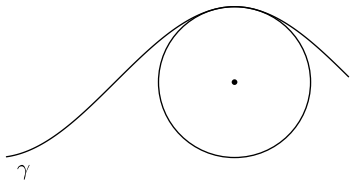}
\caption{An osculating circle of $\gamma$.}
\label{osculating}
\end{figure}

We finish this section by enunciating the fundamental theorem of the local theory of curves. Later this will inspire similar theorems for curvatures in normed planes.

\begin{teo}\label{fundamental} Let $k:[0,c]\rightarrow\mathbb{R}^2$ be a $C^1$ function. Then there exists a curve whose curvature function in the arc-length parameter is $k$. Such a curve is unique up to the choice of the initial point and the initial tangent vector.
\end{teo}
\begin{proof} Finding a curve $\gamma:[0,c]\rightarrow \mathbb{R}^2$ parametrized by arc length and with curvature $k$ is equivalent to solving the system of ordinary differential equations
\begin{align*} \left\{\begin{array}{ll} x''(s) = -k(s)y'(s) \\ y''(s) = k(s)x'(s)\end{array}\right.,
\end{align*}
where $y(s) = (x(s),y(s))$ is the coordinate representation. Indeed, if the above equalities hold, then
\begin{align*} x'(s)x''(s) = -k(s)y'(s)x'(s) = -y'(s)y''(s),
\end{align*}
and therefore $||\gamma'(s)||^2$ is constant. This is a linear first order system in $(x'(s),y'(s))$, and hence existence and uniqueness are guaranteed by the standard ODE theory when we impose the initial condition $(x'(0),y'(0)) = v$, where $v$ is a unit vector. Since the tangent field $\gamma'(s)$ is now uniquely determined, it follows (again by the standard ODE theory) that the curve is uniquely determined up to the choice of an initial point. This concludes the proof.

\end{proof}

\section{Minkowski geometry in the plane (some background)}

This section is devoted to provide some necessary background on plane \emph{Minkowski geometry}, i.e., the geometry of two-dimensional real Banach spaces. Good references to the area subject are the surveys \cite{martini1} and \cite{martini2}. For the recent state of curve theory in such planes we refer the reader to \cite{martiniandwu}. A \emph{normed} (or \emph{Minkowski}) \emph{plane} is a two dimensional vector space $X$ endowed with a norm, and we shall denote it by $(X,||\cdot||)$. We define its \emph{unit ball} to be the set $B:=\{x \in X:||x|| \leq 1\}$, and the \emph{unit circle} is the boundary $S:=\{x \in X:||x||=1\}$ of $B$. It is clear that $B$ is a \emph{convex body} (i.e., a compact, convex set with interior points) symmetric with respect to the origin, and $S$ is a \emph{convex curve}. Throughout the text we shall always assume that $S$ is of class (at least) $C^2$ and is a \emph{strictly convex curve}, i.e., it bounds a convex region without segments in the boundary. A little disclaimer regarding names: usually, the word \emph{smooth} is used to describe a unit ball supported by a unique line at each point of its boundary. Since we are assuming that $S$ is always at least of class $C^2$, it follows that all unit balls considered are smooth. When it comes to curves, the same word usually denotes the ones which are infinitely many times differentiable. Here, whenever we say that a curve is \emph{smooth} we mean that it has the sufficient differentiability to all the involved derivatives making sense.  \\

In order to study curvatures, we need notions of length of curves and area measures. The \emph{length} of a curve $\gamma:[a,b]\rightarrow X$ is defined in the usual way by
\begin{align*} l(\gamma) = \sup_P\left(\sum_{j=1}^n||\gamma(t_j)-\gamma(t_{j-1})||\right),
\end{align*}
where the supremum is taken over all partitions $P:=\{a=t_0,t_1,...,t_n=b\}$ of $[a,b]$. If $l(\gamma) < +\infty$, we say that $\gamma$ is \emph{rectifiable}. An area measure in a normed plane is obtained by fixing a \emph{determinant form} (i.e., a symplectic non-degenerate bilinear form) $[\cdot,\cdot]:X\times X \rightarrow \mathbb{R}$. The choice of such a form is unique up to constant multiplication, and it fixes an area element plus an orientation in the plane. At the moment we will not consider any normalization of the area element, unlike Busemann \cite{Bus3} and Petty \cite{Pet}.  \\

Here the first structural difference between Euclidean geometry and Minkowski geometry appears, and it implies how we will deal with curvatures: one can clearly parametrize the unit circle by arc length and by twice the area of the sector, but the parametrizations will not necessarily coincide. As we shall see a little later, the area parametrization equals the arc-length parametrization of a different norm. The second difference is that we work with an orthogonality concept which is not necessarily symmetric. Given two vectors $x, y \in X$, we say that $x$ is \emph{Birkhoff orthogonal} to $y$ (denoted $x\dashv_B y$) whenever $||x|| \leq ||x+ty||$ for any $t \in \mathbb{R}$. Geometrically, this means that the supporting line of the unit ball at $x/||x||$ is parallel to $y$. \\

In their work on Minkowski geometry, Busemann and Petty used to regard the \emph{isoperimetrix} (the solution to the isoperimetric problem in a normed plane, see \cite{Bus6}) as an important structural object. This is mainly based on the fact that the isoperimetrix is a convex body which reverses Birkhoff orthogonality. We follow the approach given in \cite{martiniantinorms}, where the isoperimetrix is considered as the unit circle of a new norm in the plane. For a normed plane with a fixed determinant form $[\cdot,\cdot]$ we define the \emph{anti-norm} to be 
\begin{align*} ||x||_a := \sup\{|[x,y]|:y \in S\},
\end{align*}
where $x \in X$. It is known that the supremum is attained for $y \in S$ if and only if $y \dashv_B x$, and it is easy to see that $||\cdot||_a$ is indeed a norm in $X$ (see \cite{martiniantinorms}). We shall denote the unit ball and the unit circle of the anti-norm by $B_a$ and $S_a$, respectively. As mentioned, the anti-norm reverses Birkhoff orthogonality. In other words, if $x \dashv_B y$, then $y \dashv_B^a x$, where $ \dashv_B^a$ denotes the Birkhoff orthogonality in the anti-norm (see Figure \ref{circleanticircle}).

\begin{figure}[h]
\centering
\includegraphics{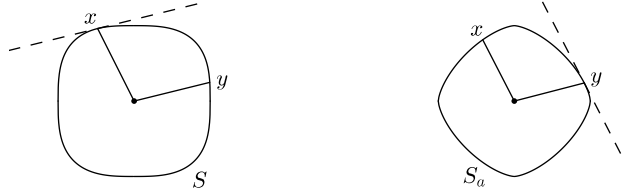}
\caption{The anti-norm reverses Birkhoff orthogonality.}
\label{circleanticircle}
\end{figure}

In the figure above, $x$ and $y$ present directions rather than vectors. This is due to the fact that Birkhoff orthogonality is homogeneous, and hence it can be regarded as a relation between directions rather than vectors. When dealing with a \emph{Radon plane} (see \cite{martiniantinorms}), we will always consider that $[\cdot,\cdot]$ is scaled in such a way that $||\cdot|| = ||\cdot||_a$ (as is well known, Radon planes are characterized as those normed planes where Birkhoff orthogonality is symmetric). \\

We denote the \emph{length of a curve $\gamma$ in the norm} by $l(\gamma)$, and \emph{in the anti-norm} by $l_a(\gamma)$. Also, the fixed determinant form induces an orientation and a standard area measure. We shall denote this area measure by $\lambda$. Let $\varphi(u):[0,2\lambda(B)]\rightarrow X$ be a positively oriented parametrization of the unit circle $S$ by twice the area of sectors (we shall use this notation throughout the text). We have
\begin{align*} u = \int_0^u[\varphi(t),\varphi'(t)] \ dt, \ \ u \in [0,2\lambda(B)].
\end{align*}
Differentiating with respect to $u$, we get
\begin{align*} 1 = [\varphi(u),\varphi'(u)] = ||\varphi(u)||\cdot||\varphi'(u)||_a = ||\varphi'(u)||_a,
\end{align*}
where the second equality follows from $\varphi(u) \dashv_B \varphi'(u)$. We observe here that a parametrization of the unit circle by means of area of sectors is a parametrization in the arc length with respect to the anti-norm. The norm and the anti-norm have a duality relation since the anti-norm of the anti-norm (in the same determinant form) is the original norm (see \cite{martiniantinorms}).  Therefore, the area parameter in the unit anti-ball equals the norm arc-length parameter in it. \\

From now on, when dealing with more than one parameter we shall reserve the symbol $'$ for the derivative with respect to the Minkowski arc length (which will be usually denoted by the letter $s$). Despite this agreement, we may use $'$ to denote the differentiation with respect to other parameters when there is no possibility of confusion.

\section{Defining curvature types in a normed plane} \label{secdefmink}

We present now the curvature concepts already defined in the literature for normed planes, and we explain how each of them is obtained by one of the perspectives described in Section \ref{curvatureeuclid}. \\

The first concept is called \emph{Minkowski curvature} by Petty in \cite{Pet}, and it is the very same that Biberstein called simply \emph{curvature} in \cite{biberstein}. This curvature concept is based on the first idea described in Section \ref{curvatureeuclid}: we study the variation of the area swept in the unit circle by the tangent vector.\\

Let $\gamma(s):[0,l(\gamma)]\rightarrow X$ be a smooth curve parametrized by arc length (in the norm). We identify the area swept by the tangent field $\gamma'$ by writing 
\begin{align}\label{tangentfield} \gamma'(s) = \varphi(u(s)), \ \ s \in [0,l(\gamma)].
\end{align}
In other words, the vector $\gamma'(s)$ is identified within the unit circle in the area sector parameter. The function $u(s):[0,l(\gamma)]\rightarrow X$ associates the length traveled on $\gamma$ with the area swept within $B$. Therefore, we define the \emph{Minkowski curvature} of $\gamma$ at $\gamma(s)$ to be
\begin{align}\label{minkowskicurvature} k_m(s) := u'(s), \ \ s \in [0,l(\gamma)].
\end{align}

In the Euclidean plane, we can also differentiate the tangent field in order to obtain the curvature. In a Minkowski plane we can do something similar, but we have more than one possible choice for the normal field (recall that Birkhoff orthogonality is not symmetric). Since the area parameter fixed in $S$ is the arc-length parameter in the anti-norm, we proceed as follows: define the \emph{right normal field} to $\gamma$ as the mapping $n_{\gamma}(s):[0,l(\gamma)]\rightarrow X$ which associates to each $s \in [0,l(\gamma)]$ the unique vector $n_{\gamma}(s)$ such that $\gamma'(s) \dashv_B n_{\gamma}(s)$ and $[\gamma'(s),n_{\gamma}(s)] = 1$. In other words, this last equality means that $n_{\gamma}(s)$ is unit in the anti-norm and that the basis $\{\gamma'(s),n_{\gamma}(s)\}$ is positively oriented. Differentiating equality (\ref{tangentfield}), we have
\begin{align} \label{frenet1} \gamma''(s) = u'(s)\frac{d\varphi}{du}(u(s)) = k_m(s)n_{\gamma}(s), \ \ s \in [0,l(\gamma)].
\end{align}
Notice that we are using the symbol $'$ to denote the derivative only with respect to $s$. Figure \ref{frenetframe} illustrates the situation.  \\

\begin{figure}[h]
\centering
\includegraphics{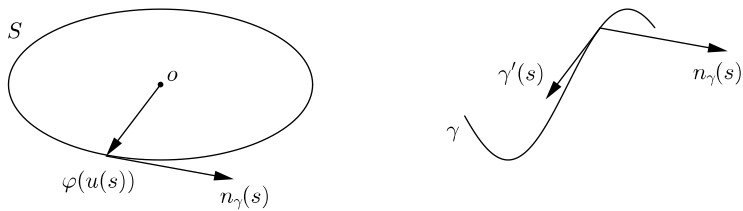}
\caption{The tangent and right normal fields of $\gamma$.}
\label{frenetframe}
\end{figure}

As long as the tangent field $\gamma'(s)$ varies through the unit circle, the right normal field $n_{\gamma}(s)$ varies through the unit anti-circle. Now these vector fields do not necessarily sweep the same area, and therefore they may yield different curvatures, unlike the Euclidean case. \\

To deal with the area swept in the unit anti-circle by the vector field $n_{\gamma}(s)$, let $\psi(v):[0,2\lambda(S_a)]\rightarrow X$ be a parametrization of the unit anti-circle by twice the area of the sector. For each $s \in [0,l(\gamma)]$ we let $v(s) \in [0,2\lambda(S_a)]$ be such that $n_{\gamma}(s) = \psi(v(s))$ and define the \emph{normal curvature} of $\gamma$ at $\gamma(s)$ to be the number
\begin{align} \label{normalcurvature} k_n(s) := v'(s), \ \ s \in [0,l(\lambda)].
\end{align}
This concept was called \emph{isoperimetric curvature} by Petty in \cite{Pet} (motivated by the fact that for any normed plane the anti-circle solves the isoperimetric problem) and \emph{anti-curvature} by Biberstein in \cite{biberstein}. Gage \cite{gage1} called it \emph{Minkowski curvature}, since he obtained, but did not work with, the concept named like this in our paper. Notice that due to duality we can re-obtain the normal curvature by differentiating $n_{\gamma}(s)$:
\begin{align}\label{frenet2} n_{\gamma}'(s) = v'(s)\frac{d\psi}{dv}(v(s)) = -k_n(s)\gamma'(s), \ \ s \in [0,l(\gamma)].
\end{align}
Notice that (\ref{frenet1}) and (\ref{frenet2}) can be regarded as \emph{Frenet formulas for a normed plane} (as pointed out in \cite{haddou}, \cite{gage1}, \cite{Pet}, and \cite{shonoda2}). From the last formula, notice carefully that $|k_n(s)| = ||n'_{\gamma}(s)||$ for all $s \in [0,l(\gamma)]$.\\

We continue now with defining \emph{circular curvature}. This concept was defined by Petty \cite{Pet}, and used later by Ghandehari \cite{Ghan1,Ghan2} and Craizer \cite{craizer}. To define it, let $\gamma(s):[0,l(\gamma)]\rightarrow X$ be a curve parametrized by arc length, and assume that $\varphi(t):[0,l(S)]\rightarrow X$ is a parametrization of the unit circle by arc length. We let $t(s)$ be the function which associates each $s \in [0,l(\gamma)]$ to the number $t(s) \in [0,l(S)]$ such that 
\begin{align*} \gamma'(s) = \frac{d\varphi}{dt}(t(s)).
\end{align*}
In other words, $t(s)$ is the length traveled along the unit circle when the vector field $\gamma'(s)$ varies as its tangent field (see Figure \ref{circcurv}).\\

\begin{figure}[h]
\centering
\includegraphics{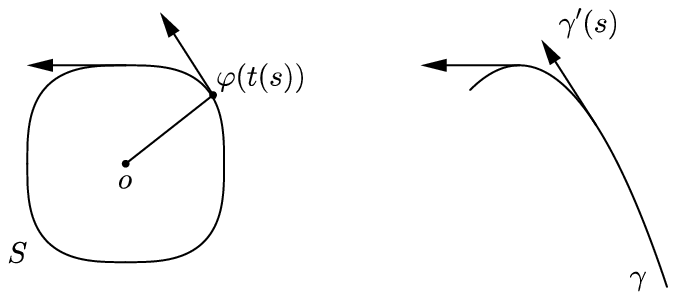}
\caption{$\gamma'(s) = \displaystyle\frac{d\varphi}{dt}(t(s))$.}
\label{circcurv}
\end{figure}

We define the \emph{circular curvature} of $\gamma$ measuring the variation of the length $t(s)$ with respect to $s$. Formally, the curvature of $\gamma$ at $\gamma(s)$ is given by
\begin{align*} k_c(s) := t'(s).
\end{align*}

In \cite{craizer}, Craizer defines the circular curvature inspired by contact geometry (see also \cite{craizermartini} for the discrete framework). The idea is to consider a Minkowski circle with second order contact in each point of the curve, and to define the curvature to be the inverse of the curvature radius. With this point of view it is also possible to define osculating circles. \\

Notice that the re-parametrization $\varphi(s):=\varphi(t(s))$ is such that the tangents to the unit circle at $\varphi(s)$ and to $\gamma$ at $\gamma(s)$ are parallel. Hence, since for each fixed $s_0 \in [0,l(\gamma)]$ we have
\begin{align}\label{eqcirc} \varphi'(s_0) = t'(s_0)\frac{d\varphi}{dt}(t(s_0)) = k_c(s_0)\gamma'(s_0),
\end{align}
it follows that the Minkowski circle 
\begin{align*} s \mapsto \gamma(s_0) - k_c(s_0)^{-1}\varphi(s_0) + k_c(s_0)^{-1}\varphi(s)
\end{align*}
has  second order contact with $\gamma$ at the point $\gamma(s_0)$. This translated circle is called the \emph{osculating circle} of $\gamma$ at $\gamma(s)$. The number $k_c(s)^{-1}$ is the \emph{curvature radius} of $\gamma$ at $\gamma(s)$. \\

It is worth saying that we have given an interpretation of circular curvature by means of contact geometry, but circular curvature in a Minkowski plane can indeed be directly defined on those lines. To do so, let $\varphi(\theta):[0,c]\rightarrow X$ be any regular parametrization of the unit circle, and let $\gamma$ be a smooth curve endowed with a parametrization $\gamma(\theta):[0,c]\rightarrow X$ such that for each $\theta \in [0,c]$ the vector $\gamma'(\theta)$ is a non-negative multiple of $\varphi'(\theta)$. In other words, we have a function $\rho:[0,c]\rightarrow[0,+\infty)$ for which
\begin{align*} \gamma'(\theta) = \rho(\theta)\varphi'(\theta), \ \ \theta \in [0,c].
\end{align*}
Of course, the number $\rho$ can be interpreted by a curvature radius, and we claim that indeed $k_c(\theta) = \rho(\theta)^{-1}$. To check this fact, let $\gamma(s):[0,l(\gamma)]\rightarrow X$ be an arc length parametrization of $\gamma$ and write
\begin{align*} \gamma'(s) = \theta'(s)\frac{d\gamma}{d\theta}(\theta(s)) = \theta'(s)\rho(\theta(s))\frac{d\varphi}{d\theta}(\theta(s)) = \rho(\theta(s))\varphi'(s).
\end{align*}
Thus, $\varphi'(s) = \rho(s)^{-1}\gamma'(s)$. From equality (\ref{eqcirc}) we get what we claimed. Notice carefully that this interpretation does not make sense for inflection points of a curve. We highlight now an interesting relation between circular and normal curvatures.

\begin{prop} The circular curvature is the normal curvature in the anti-norm.
\end{prop}
\begin{proof} Let $\gamma$ be a smooth curve, and assume that $\gamma(s):[0,l(\gamma)]\rightarrow X$ is a parametrization by arc length. As before, let $\varphi(t):[0,l(S)]\rightarrow X$ be a parametrization of the unit circle by arc length, and let $t(s):[0,l(\gamma)]\rightarrow[0,l(S)]$ be the function which associates each $s$ to the number $t(s)$ for which $\gamma'(s) = \frac{d\varphi}{dt}(t(s))$. By definition, $k_c(s) = t'(s)$ for each $s \in [0,l(\gamma)]$. \\

To obtain the normal curvature in the anti-norm, we follow the same procedure as already described, but taking the anti-norm instead of the norm as reference. Let $\gamma(s_a):[0,l_a(\gamma)]\rightarrow X$ be an (orientantion preserving) parametrization of $\gamma$ in the anti-norm arc length. First notice that since 
\begin{align*} \gamma'(s) = s_a'(s)\frac{d\gamma}{ds_a}(s_a(s)),
\end{align*}
we have that $s_a'(s) = ||\gamma'(s)||_a$. Now, for each $s_a \in [0,l_a(\gamma)]$, let $n_{\gamma}^a(s_a)$ be the unique vector such that 
\begin{align*}\frac{d\gamma}{ds_a}(s_a) \dashv_B^a n_{\gamma}^a(s_a)\ \  \mathrm{and} \ \ \left[\frac{d\gamma}{ds_a}(s_a),n_{\gamma}^a(s_a)\right] = 1.
\end{align*}
We always have that $n_{\gamma}^a(s_a) \in S$, and to obtain the normal curvature in the anti-norm we let $\varphi(u):[0,2\lambda(S)]\rightarrow X$ be a usual area parametrization of $S$, and let $u(s_a)$ be such that $n_{\gamma}^a(s_a) = \varphi(u(s_a))$. Following from the definition, the normal curvature in the anti-norm is $k_{n}^a(s_a) = \frac{du}{ds_a}(s_a)$. \\

Having described both curvatures that we want to relate, we just need two steps. First, by differentiation we obtain
\begin{align*} \frac{dn_{\gamma}^a}{ds_a}(s_a) = k_n^a(s_a)\frac{d\varphi}{du}(u(s_a)).
\end{align*}
And second, since $u$ is an anti-norm arc-length parameter, we have that
\begin{align*} \left|k_n^a(s_a)\right| = \left|\left|\frac{dn_{\gamma}^a}{ds_a}(s_a)\right|\right|_a.
\end{align*}
Since the anti-norm reverses Birkhoff orthogonality, we have that $n_{\gamma}^a(s_a) \dashv_B \frac{d\gamma}{ds_a}(s_a)$. Therefore, for each $s \in [0,l(\gamma)]$ we have that $n_{\gamma}^a(s_a(s)) = \varphi(t(s))$. Differentiating both sides, we get
\begin{align*} s_a'(s)\frac{dn_{\gamma}^a}{ds_a}(s_a(s)) = k_c(s)\frac{d\varphi}{dt}(t(s)) = k_c(s)\gamma'(s).
\end{align*}
Passing to the anti-norm in both sides, we obtain
\begin{align*} |s_a'(s)|\cdot|k_n^a(s_a)| = |k_c(s)|\cdot||\gamma'(s)||_a.
\end{align*}
Hence, since the signs of $k_c$ and $k_n^a$ are obviously the same, we have the desired equality. Notice that by duality we also have that the normal curvature is the circular curvature in the anti-norm.

\end{proof}

\begin{remark} Guggenheimer \cite{Gug1} worked with normal curvature, but he interpreted it as the circular curvature in the anti-norm. Ait-Haddou et al. \cite{haddou} obtained the Frenet formulas, but gave the same interpretation to the normal curvature. Guggenheimer called this concept simply \emph{curvature}, while Ait-Haddou used the name \emph{$\mathcal{U}$-circular curvature}.
\end{remark}

\begin{coro} The following statements are equivalent:\\

\noindent\textbf{(a)} $X$ is a Radon plane, \\

\noindent\textbf{(b)} the normal and circular curvatures of the unit circle coincide, and \\

\noindent\textbf{(c)} the normal and circular curvatures of any curve coincide.
\end{coro}
\begin{proof} Clearly, it is enough to prove \textbf{(b)}$\Rightarrow$\textbf{(a)}. Of course, the circular curvature of the unit circle is equal to $1$, and hence if \textbf{(b)} holds, we have from the Frenet formula (\ref{frenet2}) that
\begin{align*} n_{\varphi}'(s) = -\varphi'(s),
\end{align*}
where $\varphi(s):[0,l(S)]\rightarrow X$ is a usual arc length parametrization of the unit circle such that $\varphi(0)$ is an endpoint of a conjugate diameter (notice that this gives $n_{\varphi}(0) = -\varphi(0)$). By integration it follows that $||n_{\varphi}(s)||$ equals $1$. Since $||n_{\varphi}(s)||_a$ equals $1$ by definition, it follows that the norm and the anti-norm coincide in every direction. 

\end{proof}

Intuitively, one could measure curvature by considering the length on the unit circle determined by the unit tangent vector field, instead of the area (as in the definition of Minkowski curvature). Let $\varphi(t):[0,l(S)]\rightarrow X$ be an arc-length parametrization in the unit circle, and let $t(s):[0,l(\gamma)]\rightarrow[0,l(S)]$ be the function such that $\gamma'(s) = \varphi(t(s))$. Then we define the \emph{arc-length curvature} of $\gamma$ to be the number 
\begin{align*} k_l(s) := t'(s).
\end{align*}
It turns out that the \emph{arc-length curvature is precisely the Minkowski curvature in the anti-norm}. Notice that, in some sense, \emph{this is a duality relation as in the case of normal and circular curvatures}.\\

To prove our claim, let us calculate the Minkowski curvature in the anti-norm in the usual way: we assume that $\psi(v):[0,2\lambda(S_a)]\rightarrow X$ is a parametrization in the anti-circle by twice the area of sectors, and let $v(s):[0,l(\gamma)]\rightarrow[0,2\lambda(S_a)]$ be the function such that $\gamma'(s) = \psi(v(s))$. Therefore, denoting the Minkowski curvature in the anti-norm of $\gamma$ by $k_{m,a}(s)$ and differentiating $\gamma'(s)$ in two ways, we have
\begin{align*} k_l(s)\frac{d\varphi}{dt}(t(s)) = \gamma''(s) = k_{m,a}(s)\frac{d\psi}{dv}(v(s)).
\end{align*}
Then it follows that the signs of $k_l$ and $k_{m,a}$ are equal. Moreover, since $\frac{d\varphi}{dt}(t(s))$ and $\frac{d\psi}{dv}(v(s))$ are always unit vectors, we have that $|k_l(s)| = |k_{m,a}(s)|$ for any $s \in [0,l(\gamma)]$. 

\begin{remark} Thompson \cite[p. 270]{thompson} briefly defines two curvature types in a Minkowski plane. The first one (which he denoted by $\kappa$) coincides with Minkowski curvature, and the second one (which he denoted by $\kappa'$) is defined by measuring the variation of the area determined by the tangent vector field \emph{in the unit anti-circle}, but with respect to the arc length \emph{in the norm}. By our previous discussion, if we consider the variation of the area determined by the tangent vector field within the unit anti-circle with respect to the arc length \emph{in the anti-norm}, we obtain the arc-length curvature. 
\end{remark}

\section{Using an auxiliary Euclidean structure}

From the ``heuristic" point of view it would be more interesting to develop the differential geometry of curves in normed planes without considering anything else than general concepts of vector spaces. Indeed, the reader may notice that in the previous section we obtained definitions of curvature types without using even a coordinate system, although in almost all papers on Minkowski geometry in some way an auxiliary Euclidean structure is used. (Moreover, for some of the related results it even seems to be hard to get them without such a structure). Some of the formulas that we obtain in this section appeared (with slight variations) in \cite{haddou} and \cite{Pet}. \\

Any two-dimensional real vector space is isomorphic to $\mathbb{R}^2$, and Minkowski geometry is invariant through affine transformations. In other words, if a convex body $K$ is symmetric with respect to the origin and thus yields a norm in a plane $X$ (having $K$ as its unit ball), and $T:X\rightarrow X$ is an affine transformation, then the geometry derived from $T(K)$ is exactly the same (under the identification $x\mapsto Tx$, of course). Therefore, any Minkowski geometry can be studied by identifying its unit ball with a convex body $K$ centered at the origin of the plane $\mathbb{R}^2$, endowed with the usual coordinate basis $\{e_1,e_2\}$, and assuming that the fixed determinant form is the standard determinant. \\

Conversely, a convex body $K$ in $\mathbb{R}^2$ centered at the origin induces a Minkowski norm by dividing the Euclidean length of a vector by half the Euclidean length of the diameter of $K$ in the same direction of the vector. Having said this, we shall consider our unit circle $S$ parametrized in the polar form
\begin{align}\label{polar} \varphi(\theta) = p(\theta)\cdot(\cos\theta,\sin\theta), \ \ \theta \in [0,2\pi],
\end{align}
where $p:[0,2\pi]\rightarrow\mathbb{R}$ is a $\pi$-periodic function of class (at least) $C^2$. Geometrically, the polar form gives the Euclidean length of a diameter of the unit circle in a given direction. It is clear that, to obtain the unit anti-circle, one just has to consider the parametrized curve $\psi:[0,2\pi]\rightarrow\mathbb{R}^2$ given by
\begin{align} \label{anti} \psi(\theta) := -\frac{\varphi'(\theta)}{[\varphi(\theta),\varphi'(\theta)]} = -\frac{p'(\theta)}{p(\theta)^2}(\cos\theta,\sin\theta) - \frac{1}{p(\theta)}(-\sin\theta,\cos\theta), \ \ \theta \in [0,2\pi].
\end{align}
The reader may notice that, obviously, this is not a polar equation of $S_a$. We can easily derive its (Euclidean) \emph{support function} $h_{\psi}$ (which is the function that associates each oriented direction to the distance of the orthogonal tangent line of a given curve to the origin). To verify this, set $h(\theta):=p(\theta)^{-1}$ and notice that $\psi(\theta) = h'(\theta)\cdot(\cos\theta,\sin\theta) - h(\theta)\cdot(-\sin\theta,\cos\theta)$. Therefore,
\begin{align*} \psi'(\theta) = (h''(\theta) + h(\theta))\cdot(\cos\theta,\sin\theta),
\end{align*}
 and then the normal direction to $S_a$ in the point $\psi\left(\theta+\frac{\pi}{2}\right)$ is $(\cos\theta,\sin\theta)$. Since
\begin{align*} \psi\left(\theta + \frac{\pi}{2}\right) = h\left(\theta + \frac{\pi}{2}\right)\cdot(\cos\theta,\sin\theta) + h'\left(\theta+\frac{\pi}{2}\right)\cdot(-\sin\theta,\cos\theta),
\end{align*}
it follows that the support function of $S_a$ is given by $h_{\psi}(\theta) = h\left(\theta+\frac{\pi}{2}\right) = p\left(\theta+\frac{\pi}{2}\right)^{-1}$. Notice carefully that this is evaluated at the point of $S_a$ in which its normal direction is $(\cos\theta,\sin\theta)$. \\

\begin{figure}[h]
\centering
\includegraphics{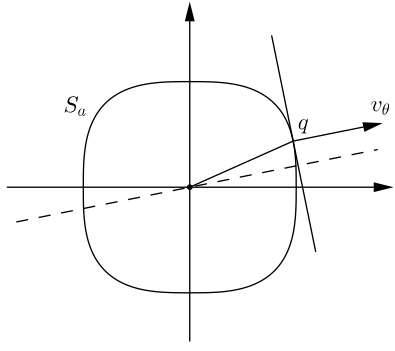}
\caption{$q$ is the point $\psi\left(\theta + \frac{\pi}{2}\right)$, where the normal direction to $S_a$ is $v_{\theta}:=(\cos\theta,\sin\theta)$.}
\label{suppanticircle}
\end{figure}

Our next goal is to obtain an expression for the Minkowski curvature of a smooth curve in terms of its Euclidean curvature (which we will denote by $k_e$), in a similar way as used by Petty in \cite{Pet}. For this sake, let $\gamma(s):[0,l(\gamma)]\rightarrow\mathbb{R}^2$ be a smooth curve parametrized by Minkowski arc length. For each $s \in [0,l(\gamma)]$ let $\theta(s) \in [0,2\pi]$ be such that $ \gamma'(s) = p(\theta(s))\cdot(\cos\theta(s),\sin\theta(s))$. Differentiating this, we get
\begin{align*} \gamma''(s) = \theta'(s)\frac{dp}{d\theta}(\theta(s))\cdot(\cos\theta(s),\sin\theta(s)) + \theta'(s)p(\theta(s))\cdot(-\sin\theta(s),\cos\theta(s)).
\end{align*}
Now, to obtain the Minkowski curvature of $\gamma$ at $\gamma(s)$, we just need to write the right normal field $n_{\gamma}(s)$ of $\gamma(s)$ (defined in Section \ref{secdefmink}) in terms of the auxiliary Euclidean structure. From (\ref{anti}) we have
\begin{align}\label{rightnormal} n_{\gamma}(s) = \frac{1}{p(\theta(s))^2}\frac{dp}{d\theta}(\theta(s))\cdot(\cos\theta(s),\sin\theta(s)) + \frac{1}{p(\theta(s))}\cdot(-\sin\theta(s),\cos\theta(s)).
\end{align}
Therefore, we may write $\gamma''(s) = \theta'(s)p(\theta(s))^2n_{\gamma}(s)$. It follows that the Minkowski curvature of $\gamma$ at $\gamma(s)$ is given by
\begin{align}\label{minkcurv} k_m(s) = \theta'(s)p(\theta(s))^2,
\end{align}
for each $s \in [0,l(\gamma)]$. \\

To determine the Euclidean curvature of $\gamma$, we let $\gamma(s_e):[0,l_e(\gamma)]\rightarrow\mathbb{R}^2$ be a parametrization by Euclidean arc length and, given $s_e \in [0,l_e(\gamma)]$, let $\theta(s_e)\in [0,2\pi]$ be the number for which 
\begin{align*}
\frac{d\gamma}{ds_e}(s_e) = (\cos\theta(s_e),\sin\theta(s_e)).
\end{align*}
Differentiating this, we get
\begin{align*} \frac{d^2\gamma}{ds_e^2}(s_e) = \frac{d\theta}{ds_e}(s_e)\cdot(-\sin\theta(s_e),\cos\theta(s_e)),
\end{align*}
and thus the Euclidean curvature of $\gamma$ at $\gamma(s_e)$ is given by $k_e(s_e) = \frac{d\theta}{ds_e}(s_e)$. To relate both curvatures, first notice that $s\mapsto s_e(s)$ is a re-parametrization of $\gamma$, and we may differentiate to obtain
\begin{align*} \gamma'(s) = s_e'(s)\frac{d\gamma}{ds_e}(s_e(s)) = s_e'(s)\cdot(\cos\theta(s_e(s)),\sin\theta(s_e(s))),
\end{align*}
from where we have $s_e'(s) = p(\theta(s))$. Now,
\begin{align}\label{eqtheta} \theta'(s) = s_e'(s)\frac{d\theta}{ds_e}(s_e(s)) = p(\theta(s))k_e(s_e(s)).
\end{align}
Substituting at (\ref{minkcurv}) it follows that the Minkowski and Euclidean curvatures of $\gamma$ at $\gamma(s)$ are related by the equality
\begin{align}\label{minkeuclidcurv} k_m(s) = k_e(s)p(\theta(s))^3, \ \ s\in [0,l(\gamma)].
\end{align}
Since the arc-length curvature is the Minkowski curvature in the anti-norm, we just have to proceed the same way to obtain an expression for it. Let $\gamma$ be parametrized by the arc length $s_a$, and assume that $q(\theta)$ is the polar equation of the unit anti-circle. Then, if $\theta(s_a)$ is the function such that $\gamma'(s_a) = q(\theta(s_a))\cdot(\cos\theta(s_a),\sin\theta(s_a))$, we have
\begin{align}\label{alceuclid} k_l(s_a) = k_e(s_a)q(\theta(s_a))^3.
\end{align}
 
To get an expression for the normal curvature, we have to derivate the right normal field expressed in (\ref{rightnormal}). For the sake of having an easier notation, we put again $h(\theta):=p(\theta)^{-1}$ and rewrite 
\begin{align*} n_{\gamma}(s) = -\frac{dh}{d\theta}(\theta(s))\cdot(\cos\theta(s),\sin\theta(s)) + h(\theta(s))\cdot(-\sin\theta(s),\cos\theta(s)). 
\end{align*}
Differentiating with respect to $s$, we have

\begin{align*} n'_{\gamma}(s) = \theta'(s)\left(-\frac{d^2h}{d\theta^2}(\theta(s)) - h(\theta(s))\right)\cdot(\cos\theta(s),\sin\theta(s)) =\\= \frac{\theta'(s)}{p(\theta(s))}\left(-\frac{d^2h}{d\theta^2}(\theta(s)) - h(\theta(s))\right)\gamma'(s) = k_e(s)\left(-\frac{d^2h}{d\theta^2}(\theta(s)) - h(\theta(s))\right)\gamma'(s),
\end{align*}
where the last equality comes from (\ref{eqtheta}). It follows from (\ref{frenet2}) that
\begin{align*} k_n(s) = k_e(s)\left(\frac{d^2h}{d\theta^2}(\theta(s)) + h(\theta(s))\right).
\end{align*}

Recall now that the unit anti-circle is parametrized by $\psi(\theta) = h'(\theta)\cdot(\cos\theta,\sin\theta) - h(\theta)\cdot(-\sin\theta,\cos\theta)$. The geometric interpretation of this parametrization is that $\psi'(\theta)$ points in the direction of $(\cos\theta,\sin\theta)$. Hence, the tangent line to the unit anti-circle at $\psi(\theta(s))$ is parallel to $\gamma'(s)$. The Euclidean curvature of the unit anti-circle at $\psi(\theta)$ is easily calculated, namely
\begin{align*} k_{\psi}(\theta) = \frac{[\psi'(\theta),\psi''(\theta)]}{||\psi'(\theta)||_e^3} = \left(\frac{d^2h}{d\theta^2}(\theta)+h(\theta)\right)^{-1},
\end{align*}
where $||\cdot||_e$ denotes the Euclidean norm. Therefore, we have
\begin{align}\label{normaleuclid} k_n(s) = \frac{k_e(s)}{k_{\psi}(\theta(s))}.
\end{align}
In other words, the normal curvature of a curve $\gamma$ at $\gamma(s)$ is the ratio between the Euclidean curvature of $\gamma$ at $\gamma(s)$ and the Euclidean curvature of the unit anti-circle at a point where the direction $\gamma'(s)$ supports it.\\

Finally, since the circular curvature is the normal curvature in the anti-norm, we have that the circular curvature of $\gamma$ at $\gamma(s)$ is given by
\begin{align}\label{circulareuclid} k_c(s) = \frac{k_e(s)}{k_{\varphi}(\theta(s))},
\end{align} 
where $k_{\varphi}(\theta(s))$ is the Euclidean curvature of the unit circle at a point where the direction $\gamma'(s)$ supports it. The reader may notice that this was the definition used by Ghandehari in \cite{Ghan1}.\\

In \cite{Pet}, Petty mentioned that if the Minkowski curvature and the normal curvature coincide for any curve, then the geometry is Euclidean. We now present a formal proof of a similar result. 

\begin{teo}\label{minknormequal} Let $(X,||\cdot||)$ be a normed plane. If there exists a smooth curve $\gamma:[0,c]\rightarrow X$ whose tangent field goes through every direction of the plane, and whose Minkowski curvature and normal curvature coincide in each point, then the norm is Euclidean.
\end{teo}
\begin{proof} If such a curve $\gamma$ exists, then the equalities (\ref{minkeuclidcurv}) and (\ref{normaleuclid}) yield that the function $h(\theta) = p(\theta)^{-1}$ restricted to $[0,\pi]$ is a solution to the problem
\begin{align*} \left\{\begin{array}{ll}h''(\theta) + h(\theta) = h(\theta)^{-3}, \ \ \theta \in [0,\pi] \\ h(0) = h(\pi), \ \ h'(0) = h'(\pi) \end{array}\right.,
\end{align*}
which is an Ermakov-Pinney equation with boundary conditions (see \cite{pinney}). The function identical to $1$ is a solution, and it must be unique since another solution would give a non-zero solution to the problem
\begin{align*} \left\{\begin{array}{ll}y''(\theta) + y(\theta) = 0, \ \ \theta \in [0,\pi] \\ y(0) = y(\pi), \ \ y'(0) = y'(\pi) \end{array}\right..
\end{align*}
It follows that $p$ equals $1$, and hence the unit circle is the Euclidean circle. The reader can find more on the relations between these equations in \cite{pinney}.

\end{proof}

\begin{coro} The same conclusion holds under the same hypothesis if we replace Minkowski curvature and normal curvature by, respectively, arc-length curvature and circular curvature. 
\end{coro}

NameIy, in this case the anti-norm is Euclidean, and so is the norm.

\section{Curves of constant curvature}

In the Euclidean plane, a curve with constant curvature must be an arc of a circle. We will study \emph{curves of constant curvature in Minkowski planes}, considering the three curvature types that we defined. We begin with identifying the curves which have constant circular and normal curvatures, since we will not use an auxiliary Euclidean structure to do so. This is not the case for the Minkowski curvature, and for this reason we shall deal with it later.

\begin{teo} A curve $\gamma:[0,c]\rightarrow X$ has constant circular curvature if and only if it is an arc of a Minkowski circle.
\end{teo}
\begin{proof} It is clearly enough to consider that the circular curvature $k_c$ equals $1$. In this case, we proceed as usual and consider that $\varphi(t)$ is an arc length parametrization of the unit circle, and that $t(s)$ is the function such that $t(0) = 0$ and $\gamma'(s) = \frac{d\varphi}{dt}(t(s))$ for each $s \in [0,c]$. Since $k_c(s) = t'(s) = 1$ for each $s \in [0,c]$, it follows that
\begin{align*} t(s) = t(0) + \int_0^st(u) \ du = s,
\end{align*}
and hence $\varphi'(s) = \gamma'(s)$ for every $s \in [0,c]$. Therefore, $\gamma$ is an arc of a unit circle. The converse is obvious.

\end{proof}

\begin{coro} A curve in a normed plane has constant normal curvature if and only if it is an arc of an anti-circle.
\end{coro}
\begin{proof} Since the circular curvature is the normal curvature in the anti-norm, any curve of constant normal curvature must have constant circular curvature in the anti-norm. Therefore, from the previous theorem it follows that such a curve must be an arc of an anti-circle. 

\end{proof}

It is known that if all the normal lines of a curve in the Euclidean plane intersect at one point, then this curve must be contained in a circle. In a more general setting, hypersurfaces whose \emph{affine normal lines} all meet at one point are well-studied in the field of affine differential geometry (see \cite{nomizu}). We now classify the curves of a normed plane whose right and left-normal lines respectively meet all at one point. 

\begin{prop} Let $\gamma:[0,c]\rightarrow X$ be a curve of class $C^2$. If all the left-normal lines of $\gamma$ meet at a point, then $\gamma$ is contained in a circle. Analogously, if all the right normal lines of $\gamma$ meet at a point, then the $\gamma$ is contained in an anti-circle. 
\end{prop}
\begin{proof} Assume that $\gamma$ is parametrized by arc length. If $\varphi(t)$ is, as usual, an arc length parametrization of the unit circle, and $t(s)$ is, also as usual, the function such that $\gamma'(s) = \frac{d\varphi}{dt}(t(s))$, then the hypothesis on the left-normal lines yields that we may write
\begin{align*} \gamma(s) = p + f(s)\varphi(t(s)),
\end{align*}
for some function $f:[0,c]\rightarrow\mathbb{R}$. Therefore,
\begin{align*} \frac{d\varphi}{dt}(t(s)) = \gamma'(s) = f'(s)\varphi(t(s)) + f(s)t'(s)\frac{d\varphi}{dt}(t(s)).
\end{align*}
Hence, we must have $f'(s) = 0$ and $f(s) = k_c(s)^{-1}$. It follows that the circular curvature is constant, and then $\gamma$ is contained in a circle. For the right normal lines we just have to work with the anti-circle instead of the circle. 

\end{proof}

To obtain the curves of constant Minkowski curvature we follow Petty \cite{Pet}. Let $\gamma:[0,c]\rightarrow X$ be a curve whose Minkowski curvature equals $1$ (the other cases are clearly homothets). From (\ref{minkeuclidcurv}) it follows that the Euclidean curvature of $\gamma$ is given by $k_e(s) = p(\theta(s))^3$.  Recall that $\theta(s)$ is the direction of the tangent vector to $\gamma$ at $\gamma(s)$. Therefore, at the point where the normal to $\gamma$ points in the direction of the vector $(\cos\theta,\sin\theta)$, the Euclidean curvature is given by $p(\theta+\pi/2)^3$.  Since this curvature function is positively defined over the interval $[0,2\pi]$ and is $\pi$-periodic, we may search a curve which is closed, convex and (up to a translation) symmetric with respect to the origin. The Euclidean curvature of such a curve at the point where the normal points in the direction $(\cos\theta,\sin\theta)$ is known to be given in terms of its support function $h_{\gamma}:[0,2\pi]\rightarrow X$ by $[h_{\gamma}''(\theta) + h_{\gamma}(\theta)]^{-1}$. Hence this support function must be a solution of the boundary value Sturm-Liouville problem (see \cite{coddington})
\begin{align*} \left\{\begin{array}{ll} h_{\gamma}''(\theta) + h_{\gamma}(\theta) = p(\theta + \pi/2)^3, \ \ \theta \in [0,\pi] \\ h_{\gamma}(0) = h_{\gamma}(\pi), \ \   h'_{\gamma}(0) = h'_{\gamma}(\pi) \end{array}\right.,
\end{align*}
for which existence and uniqueness are guaranteed. A calculation shows that the solution is given by
\begin{align*} h_{\gamma}(\theta) = \frac{1}{2}\int_{\theta}^{\theta + \pi}p(u+\pi/2)^3\sin(u-\theta) \ du.
\end{align*}

Petty \cite{Pet} observed that the curve with support function as above is a homothet of the so-called \emph{centroid curve}, which consists of all loci of the centroids of the regions obtained by cutting the unit circle with lines through the origin. It is clear that the curves of constant arc-length curvature will be the homothetic to the centroid curve of the unit anti-circle. \\

In \cite{Pet} there is a proof of the fact that if an anti-circle has constant Minkowski curvature, then the plane is Euclidean. There it is also mentioned that if a Minkowski circle has constant Minkowski curvature, then the plane must be Euclidean as well. The first result can be regarded as a consequence of Theorem \ref{minknormequal}. Indeed, if the unit anti-circle has constant Minkowski curvature, then, up to rescaling, we may assume that the determinant form is unit. But the unit anti-circle has also constant unit normal curvature, and then we have the desired. Petty did not present a proof for the second result, and so we do it now.

\begin{teo} If the unit circle of a normed plane has constant Minkowski curvature, then the plane is Euclidean. The same holds if the unit anti-circle has constant arc-length curvature.
\end{teo}
\begin{proof} Notice first that, up to rescaling the determinant form $[\cdot,\cdot]$, we may assume that the Minkowski curvature is unit. Consider that $\varphi(s):[0,l(S)]\rightarrow X$ is an arc-length parametrization of the unit circle, and let $u(s):[0,l(S)]\rightarrow[0,2\lambda(S)]$ be the function such that $\varphi'(s) = \varphi(u(s))$, where $u$ is the usual area parameter in $S$. Since $k_m(s)$ equals $1$, it follows that $u'(s)$ is equal to $1$, and hence $u(s) = s + c$ for some constant $c \in \mathbb{R}$. The fact that the area and arc-length parameters are proportional gives that the unit circle is an \emph{equiframed curve}, and since $S$ is smooth it follows that it must be a \emph{Radon curve} (see \cite{martiniantinorms} and \cite{martiniequiframed}). Even more, if we write
\begin{align*} \varphi'(s) = \varphi(u(s)) = \varphi(s+c), \ \ s \in [0,l(S)],
\end{align*}
then $\varphi(s) \dashv_B \varphi(s+c)$ for any $s \in [0,l(S)]$. Geometrically this means that to find the right Birkhoff orthogonal direction to a given vector we always travel the same arc length over the unit circle. Since the plane is Radon we have that $-\varphi(s) = \varphi''(s) = \varphi(s+2c)$, and hence we may consider $c = l(S)/4$. Using again the proportionality between area and arc length in the unit circle, and remembering the uniqueness of Birkhoff orthogonality in both sides (the considered normed plane is smooth and strictly convex) we have that two non-zero vectors are Birkhoff orthogonal if and only if the diameters in their directions divide the domain bounded by the unit circle into four pieces of equal areas. According to \cite[Proposition 3]{alonso1997area} this is a characterization of the Euclidean plane (see also \cite{Ma-Wu2010}). The other claim comes straightforwardly replacing the norm by the anti-norm. 

\end{proof}

We mention a further result in this direction: \emph{if an anti-circle has constant circular curvature, then the plane is Radon}. Indeed, if an anti-circle has constant circular curvature, then it has constant normal curvature in the anti-norm. Hence it must be a circle. Since a circle and an anti-circle are homothets, it follows that the plane is Radon. By duality, the same holds if a circle has constant normal curvature. We finish this section with the simple observation that a \emph{curve for which a curvature type is equal to $0$ must be a straight line segment}. This is immediately noticed when using an auxiliary Euclidean structure, since a curvature type vanishes if and only if the Euclidean curvature vanishes.

\section{The fundamental theorem and invariance under isometries}

Using the results of the previous section we can easily derive fundamental theorems (in the sense of Theorem \ref{fundamental}) for the curvature concepts in a normed plane. The first part of this section refers to this topic.

\begin{teo} Let $k:[0,c]\rightarrow\mathbb{R}$ be a $C^1$ function, and let $(X,||\cdot||)$ be a normed plane. Then there exists a curve $\gamma:[0,c]\rightarrow X$ whose Minkowski curvature function in the arc length parameter is $k$. Such a curve is unique up to the choice of the initial point and the initial tangent vector. The same holds for the normal curvature, for the circular curvature, and for the arc-length curvature. 
\end{teo}
\begin{proof} In view of  (\ref{minkeuclidcurv}), prescribing the Minkowski curvature yields the Euclidean curvature. Hence, we just have to apply Theorem \ref{fundamental}. For the normal and circular curvature types the argument is, in view of (\ref{normaleuclid}) and (\ref{circulareuclid}), precisely the same. For the arc-length curvature we just have to use (\ref{alceuclid}) together with a re-parametrization.

\end{proof}

In \cite{peri}, this theorem was extended for Biberstein's anti-curvature to
planes with convex unit balls which no longer need to be symmetric at the origin.\\

An \emph{isometry} between normed planes $(X,||\cdot||_X)$ and $(Y,||\cdot||_Y)$ is a surjective map $T:X\rightarrow Y$ such that $||Tu - Tv||_Y = ||u-v||_X$ for any $u,v \in X$. Any isometry of the Euclidean plane can be presented as combination of (translations, rotations, and) reflections. Regarding isometries of a normed plane, the Mazur-Ulam Theorem guarantees that they are linear up to translation (see \cite[Theorem 3.1.2]{thompson}), and so we will always work with linear isometries. In particular, any isometry of a normed plane is either \emph{orientation preserving}, namely when the sign of the fixed determinant form does not change under its action, or \emph{orientation reversing} otherwise. We also have the following important result.

\begin{lemma}\label{isometryantinorm} Let $(X,||\cdot||)$ be a normed plane with fixed determinant form $[\cdot,\cdot]$ and associated anti-norm $||\cdot||_a:X\rightarrow\mathbb{R}$, and let $T:X\rightarrow X$ be an isometry. Then $T$ is an isometry in the anti-norm.
\end{lemma}
\begin{proof} Since $T$ is an isometry, it maps the unit ball onto the unit ball. Therefore, the absolute value of its determinant is $1$, and so it leaves the absolute value of the determinant form invariant: $|[x,y]| = |[T(x),T(y)]|$ for any $x,y \in X$. Because $T$ is an isometry, it leaves all properties of X, that are defined only in terms of the norm, invariant. Thus, if $B_a$ is the unit ball of the anti-norm, then $T(B_a)$ is also the unit ball of the anti-norm (not just any anti-norm ball, but the unit anti-ball, because $|[\cdot,\cdot]|$ stays invariant). This yields $T(B_a)=B_a$, which means that $T$ is also an isometry for the anti-norm.

\end{proof}

\begin{remark} Since the anti-norm of the anti-norm is the original norm, it follows that an isometry of the anti-norm must be also an isometry of the norm. \\
\end{remark}

It is clear that in the Euclidean plane two curves have the same curvature if and only if they are equal up to a \emph{rigid motion} (which we define here to be the composition of a translation with a rotation). Geometrically, they have the same shape, and there is a composition of a translation (which relates their respective initial points) and a rotation (relating their corresponding tangent vectors) yielding a coinciding position of them. When it comes to reflections, Euclidean curvature changes its sign, but its absolute value is preserved. In other words, an isometry of the Euclidean plane acts on the curvature of a curve preserving its absolute value and changing or maintaining its sign according to its orientation. This discussion raises the analogous question: what is the action of an isometry of a Minkowski plane on the curvatures of a given curve? In addition one may ask: what can be said on curves having (one or more) equal curvature(s)? 

\begin{teo} The Minkowski, normal, circular and arc-length curvatures of a curve are, up to the sign, invariant under an isometry of the considered normed plane. Their signs are preserved if and only if the isometry preserves orientation. 
\end{teo}
\begin{proof} Let $\gamma(s):[0,l(\gamma)]\rightarrow X$ be a curve parametrized by arc length, and let $T:X\rightarrow X$ be an isometry. Let $\sigma(s):[0,l(\gamma)]\rightarrow X$ be the image curve $\sigma(s) = T(\gamma(s))$. We start with the Minkowski curvature, which we shall denote by $k_{m,\gamma}$ and $k_{m,\sigma}$ for $\gamma$ and $\sigma$, respectively. Notice that since $T$ is an isometry, $s$ is an arc-length parameter of $\sigma$ as well. Therefore, in order to obtain the Minkowski curvatures of the curves we shall proceed as usual and let $\varphi(u):[0,2\lambda(S)]\rightarrow X$ be an area parametrization of the unit circle. Next, we let $u_{\gamma}, u_{\sigma}:[0,l(\gamma)]\rightarrow\mathbb{R}$ be functions such that $\gamma'(s) = \varphi(u_{\gamma}(s))$ and $\sigma'(s) = \varphi(u_{\sigma}(s))$ for any $s \in [0,l(\gamma)]$. Since $T$ is linear, we have
\begin{align*} T(\gamma''(s)) = \sigma''(s) = u_{\sigma}'(s)\frac{d\varphi}{du}(u_{\sigma}(s)) = u_{\sigma}'(s)n_{\sigma}(s).
\end{align*}
On the other hand, since $\gamma''(s) = u_{\gamma}'(s)n_{\gamma}(s)$, we have the equality
\begin{align} \label{eqcurvs} u_{\gamma}'(s)T(n_{\gamma}(s)) = u_{\sigma}'(s)n_{\sigma}(s).
\end{align}
From Lemma \ref{isometryantinorm} we have $||T(n_{\gamma}(s))||_a = ||n_{\gamma}(s)||_a = 1$. It follows that
\begin{align*} |u'_{\gamma}(s)| = |u'_{\sigma}(s)|,
\end{align*}
that is, $|k_{m,\sigma}(s)| = |k_{m,\gamma}(s)|$. To verify that the sign is preserved if and only if $T$ is orientation preserving, one just has to notice that
\begin{align*} 1 = [\sigma'(s),n_{\sigma}(s)] = \pm[T(\gamma'(s)),T(n_{\gamma}(s))],
\end{align*}
with the same sign for any $s \in [0,l(\gamma)]$, where the positive one appears when $T$ is orientation preserving, and the negative one occurs otherwise. Therefore, $n_{\sigma}(s)$ has the same orientation as $T(n_{\gamma}(s))$ if and only if $T$ is orientation preserving. Now equality (\ref{eqcurvs}) gives the desired regarding signs.\\

Let now $k_{n,\gamma}$ and $k_{n,\sigma}$ be the normal curvatures of $\gamma$ and $\sigma$, respectively. We just write down the equations
\begin{align*} n'_{\gamma}(s) = -k_{n,\gamma}(s)\gamma'(s) \  \mathrm{and} \\ n_{\sigma}'(s) = -k_{n,\sigma}(s)\sigma'(s).
\end{align*}
Since $\sigma'(s) = T(\gamma'(s))$ and $n_{\sigma}(s) = \pm T(n_{\gamma}(s))$, where the sign depends on whether or not $T$ is orientation preserving, it follows that
\begin{align*} -k_{n,\sigma}(s)\sigma'(s) = n'_{\sigma}(s) = \pm T(n_{\gamma}(s)) = \pm k_{n,\gamma}(s)T(\gamma'(s)) = \pm k_{n,\gamma}(s)\sigma'(s).
\end{align*}
Therefore $k_{n,\sigma}(s) = \pm k_{n,\gamma}(s)$, where the positive sign stands if and only if $T$ is orientation preserving. \\

For the arc-length curvature and the circular curvature, just recall that they are, respectively, the Minkowski curvature and the normal curvature in the anti-norm. Since an isometry of the norm is also an isometry of the anti-norm, we have that the absolute value of the circular curvature is invariant under an isometry. For the sign, one just has to notice that the orientation of a linear map is not depending on the norm of the plane. 

\end{proof}
 
We deal now with the converse question: if two curves have the same curvature function, does then exist an isometry that carries one onto the other? The answer is no for each curvature type. To provide counterexamples for circular and normal curvatures, the idea is to consider portions of the unit circle and unit anti-circle, respectively, which do not \emph{coincide} (i.e., cover each other) under a linear isometry. For this sake, we need the following result (which is interesting for itself, since it is a metric characterization of the Euclidean plane among all Minkowski planes). 

\begin{prop} If any two arcs of equal length of the unit circle of a normed plane coincide under some orientation preserving linear isometry, then the plane is Euclidean. 
\end{prop}
\begin{proof} First, notice that a plane with this property must be Radon, since the area of a sector of the unit ball must be proportional to the arc length of the corresponding arc of the unit circle. Now notice that if any two arcs of equal length can be carried one onto the other under an orientation preserving linear isometry, then any two points of the unit circle can be carried one onto the other by a linear isometry. Denote by $b(\cdot)$ the right Birkhoff orthogonal unit vector such that $[\cdot,b(\cdot)] > 0$. Therefore, given $x,y \in S$, any orientation preserving isometry $T$ such that $T(x) = y$ must accomplish $T(b(x)) = b(y)$, since any isometry preserves Birkhoff orthogonality. \\

It follows that, for any $x,y \in S$, the linear mapping defined by the equalities $T(x) = y$ and $T(b(x)) = b(y)$ is an isometry. We shall use this fact to prove the desired result with the help of a known characterization of the Euclidean plane. We say that two vectors $x,y \in X$ are \emph{isosceles orthogonal} (denoted $x \dashv_I y$) whenever $||x+y|| = ||x-y||$. If any Birkhoff orthogonal unit vectors are also isosceles orthogonal, then the plane is Euclidean (see \cite{alonso}). Let, then, $x,b(x) \in S$ be any pair of Birkhoff orthogonal vectors. The mapping defined by the equalities $T(x) = b(x)$ and $T(b(x)) = b^2(x)$ is an isometry, and since the plane is Radon, we have $b^2(x) = -x$. Therefore,

\begin{align*} ||x + b(x)|| = ||T(x+b(x))|| = ||b(x) - x||.
\end{align*}
Hence Birkhoff orthogonality of unit vectors indeed implies isosceles orthogonality of these vectors. This concludes the proof. 

\end{proof}

As we have mentioned, this proposition guarantees that in any non-Euclidean Minkowski plane we must have curves with equal normal or circular curvatures which do not coincide under an isometry of the plane. Indeed, any arc of the unit circle has circular curvature equal to $1$, and hence we just have to choose two arcs which cannot be carried one onto the other by an isometry. For the normal curvature we do the same, but replacing the unit circle by the unit anti-circle. \\

As a matter of fact, the group of isometries of a Minkowski plane is typically ``narrow"; see, e.g., pp. 75--76 from \cite{thompson}, and \cite{Ma-Spi-Stra}. For that reason, the task of finding curves with the same curvature which do not coincide under an isometry is not difficult. Take, for example, any usual $l_p$ norm in $\mathbb{R}^2$, with $2 < p < +\infty$, whose unit ball is illustrated below. \\

\begin{figure}[h]
\centering
\includegraphics{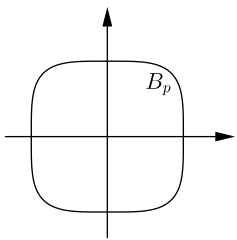}
\caption{The unit ball $B_p$ of an $l_p$ norm.}
\label{lpnorm}
\end{figure} 

It is easy to see that, if $2 < p < +\infty$, then the isometries of an $l_p$ plane are compositions of translations, reflections and rotations by multiples of $\pi/4$. Therefore, any two curves with the same curvature (any type), but whose initial tangent vectors form an angle which is not a multiple of $\pi/4$, do not coincide under an isometry of the plane. 

\section{Extremal curvature values and constant width}

There are many results regarding the curvature function of a closed, convex curve (which we will always assume to be at least of class $C^2$) in the Euclidean plane and its geometric properties. As a simple example, it is well known that convexity itself is related to the curvature function. Namely, a simple, closed curve is strictly convex if and only if its Euclidean curvature is (up to re-orientation of the parameter) strictly positive. This can be easily carried over to a normed plane: indeed, if any curvature type is strictly positive, then the Euclidean curvature is strictly positive. This is an immediate consequence of the equalities (\ref{minkeuclidcurv}), (\ref{normaleuclid}), and (\ref{circulareuclid}). \\

The \emph{four vertex theorem} states that the Euclidean curvature function of any closed, convex curve must have at least four local extrema (for a nice exposition of the topic see \cite{deturck}). This is possibly the most known and studied result regarding the curvature function of closed, convex plane curves. Versions of the four vertex theorem for normed planes were already studied. Heil \cite{heil} reduced the four vertex theorem in relative and Minkowski geometry to a generalization of a theorem on lacunary Fourier series. In \cite{tabachnikov}, Tabachnikov dealt with the circular curvature adopting a contact geometry point of view. In the introduction of this paper, one can also find some older references for the four vertex theorem in normed planes. Petty \cite{Pet} proved the four vertex theorem for Minkowski curvature, and we claim here that his method can be used to prove this theorem also for normal and circular curvatures. The chief ingredient of the proof is the following theorem.

\begin{teo}[Wilhelm S\"{u}ss, \cite{suss}] Two distinct points of a simple, closed and convex curve are said to be \emph{opposite} if the tangent lines to the curve at them are parallel. A simple, closed and strictly convex curve has at least three pairs of opposite points with the same Euclidean curvature. 
\end{teo}

The Four Vertex Theorem for the Euclidean plane follows from this theorem by standard analysis arguments. Now the Four Vertex Theorems for all curvature types in a normed plane follows from an easy observation.

\begin{teo}[The four vertex theorem for normed planes] Let $\gamma:[0,c]\rightarrow X$ be a simple, closed and strictly convex curve in a normed plane. Then each curvature function of $\gamma$ has at least four local extrema. 
\end{teo}
\begin{proof} From the equalities (\ref{minkeuclidcurv}), (\ref{alceuclid}), (\ref{normaleuclid}), and (\ref{circulareuclid}) it follows immediately that $\gamma$ has at least three pairs of opposite points with respectively equal Minkowski (or normal, or circular) curvature. Indeed, these pairs will be exactly the same as the ones for the Euclidean curvature. Thus, from standard analysis arguments we have that each curvature function of $\gamma$ has at least four local extrema.

\end{proof}

Since we are working with curves and unit circles of class $C^2$, the circular curvature function of a simple, closed, strictly convex curve must have absolute maximum and minimum.  They are clearly associated to the smallest and to the largest circles of curvature (in the sense of curvature radius), respectively. The points of maximum and minimum \emph{normal} curvature are, of course, associated to the smallest and the largest \emph{anti-circles} of curvature, since the normal curvature is the circular curvature in the anti-norm. The natural questions that appear are: Does the region bounded by a curve contain its smallest circle of curvature? Is a curve always contained within the region bounded by its largest curvature circle? Guggenheimer \cite{Gug1} tackled those questions, and we shall obtain similar results as a consequence of the following, more general (but also intuitive) theorem.

\begin{teo}\label{inclusion} Let $\gamma,\sigma:[0,2\pi]\rightarrow X$ be simple, closed and strictly convex curves of class $C^2$ which have the same initial point and the same initial tangent direction. Assume that these curves are parametrized by the (oriented) angle of their respective tangent vectors with some fixed direction, and that they are positively oriented. If $k_{c,\gamma}(\theta) \leq k_{c,\sigma}(\theta)$ for each $\theta \in [0,2\pi]$ , then the curve $\sigma$ is contained in the region bounded by the curve $\gamma$. 
\end{teo}
\begin{proof} The strategy of the proof is to show that the (Euclidean) support function of $\gamma$ is always greater than or equal to the support function of $\sigma$. For the sake of simplicity, we shall adopt an auxiliary Euclidean structure positioned in such a way that the origin lies within the region bounded by $\sigma$, and the vector $(1,0)$ is the unit normal to $\gamma'(0)$. If $h_{\gamma}(\theta)$ and $h_{\sigma}(\theta)$ are the respective support functions of $\gamma$ and $\sigma$ (where the vector $(\cos\theta,\sin\theta)$ gives the outward pointing normal direction, as usual), then we may write
\begin{align*} \gamma(\theta) = h_{\gamma}(\theta)\cdot(\cos\theta,\sin\theta)+h_{\gamma}'(\theta)\cdot(-\sin\theta,\cos\theta) \ \ \mathrm{and} \\ 
\sigma(\theta) = h_{\sigma}(\theta)\cdot(\cos\theta,\sin\theta) + h_{\sigma}'(\theta)\cdot(-\sin\theta,\cos\theta),
\end{align*}
for $\theta \in [0,2\pi]$. Notice that since $\gamma(0) = \sigma(0)$ and both curves have the same normal direction at this point, it follows that 
\begin{align*} h_{\gamma}(0) = h_{\sigma}(0) \ \ \mathrm{and}\\
h'_{\gamma}(0) = h'_{\sigma}(0).
\end{align*}
From (\ref{circulareuclid}) we have that the \emph{Euclidean} curvatures of $\gamma$ and $\sigma$ are given by
\begin{align*} k_{e,\gamma}(\theta) = k_{\varphi}(\theta)k_{c,\gamma}(\theta) \ \ \mathrm{and} \\
k_{e,\sigma}(\theta) = k_{\varphi}(\theta)k_{c,\sigma}(\theta),
\end{align*} 
respectively, where $k_{\varphi}(\theta)$ is the Euclidean curvature of the unit circle at the point whose normal unit vector is $(\cos\theta,\sin\theta)$. Therefore, the support functions of $\gamma$ and $\sigma$ satisfy the respective differential equations
\begin{align*} h''_{\gamma}(\theta) + h_{\gamma}(\theta) = [k_{\varphi}(\theta)k_{c,\gamma}(\theta)]^{-1} \ \ \mathrm{and} \\
h''_{\sigma}(\theta) + h_{\sigma}(\theta) = [k_{\varphi}(\theta)k_{c,\sigma}(\theta)]^{-1}, 
\end{align*}
with equal initial conditions. Solving by variation of parameters, we can write the solution to the first equation as
\begin{align*} h_{\gamma}(\theta) = h_{\gamma}(0)\cos\theta + h'_{\gamma}(\theta)\sin\theta + \int_0^{\theta}[k_{\varphi}(u)k_{c,\gamma}(u)]^{-1}\sin(\theta - u) \ du,
\end{align*}
and the other equation has a completely analogous solution. Since the initial conditions coincide, we may write
\begin{align*} h_{\gamma}(\theta) - h_{\sigma}(\theta) = \int_0^{\theta}\left([k_{\varphi}(u)k_{c,\gamma}(u)]^{-1}-[k_{\varphi}(u)k_{c,\sigma}(u)]^{-1}\right)\sin(\theta - u) \ du \geq 0
\end{align*}
for every $\theta \in [0,\pi]$, where the inequality comes from the hypothesis on the circular curvatures. Now we notice that reversing the orientation of the curves and repeating the argument, we have the same result for the other portion of the curves (see Figure \ref{enclosedcurve}). Therefore, we have that $h_{\gamma}(\theta) \geq h_{\sigma}(\theta)$ for every $\theta \in [0,2\pi]$, and hence the region bounded by $\gamma$ contains the curve $\sigma$.  

\end{proof}

\begin{figure}[h]
\centering
\includegraphics{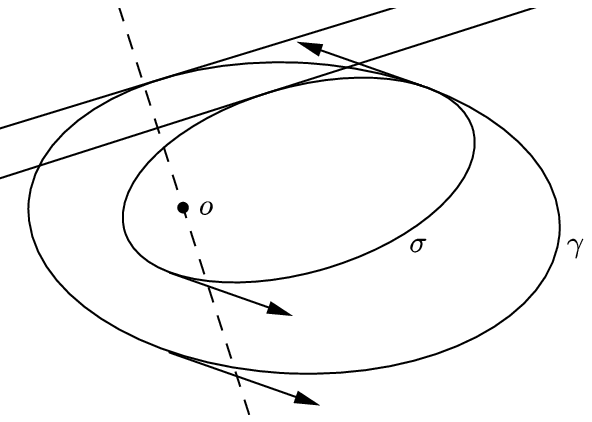}
\caption{Theorem \ref{inclusion}.}
\label{enclosedcurve}
\end{figure}

\begin{remark} Clearly, the theorem also holds if we consider the Minkowski curvature or the normal curvature instead of the circular curvature. Indeed, use (\ref{minkeuclidcurv}) and (\ref{normaleuclid}), respectively, instead of (\ref{circulareuclid}), and then the proof would be analogous.
\end{remark}

\begin{coro} Let $\gamma:[0,c]\rightarrow X$ be a simple, closed and strictly convex curve of class $C^2$. Then the region bounded by $\gamma$ contains its smallest circle and anti-circle of curvature. Moreover, $\gamma$ is contained in both regions bounded by its largest circle and anti-circle of curvature, respectively.
\end{coro}
\begin{proof} Due to Theorem \ref{inclusion}, the proof is straightforward.

\end{proof}

It is worth mentioning that, by different methods, Guggenheimer \cite{Gug1} proved that, under the same hypothesis, the region bounded by the curve contains at least two of its circles of curvature. Also, Ghandehari \cite{Ghan1} enunciated a related result as consequence of a theorem concerning the shape of the curves which have minimum length under certain conditions (namely, referring to initial, and final points, initial tangent direction, and constrained circular curvature). \\

We finish this section by revisiting a result by Petty \cite{Pet} on the important class of curves of \emph{constant width}. Let $\gamma:[0,c]\rightarrow X$ be a closed, simple and strictly convex curve. The \emph{width} of $\gamma$ in a given direction is the (Minkowski) distance between the two tangent lines to $\gamma$ at this direction (see \cite{martini2}). For a curve whose width is constant, it is known that the sum of the curvature radii of any pair of opposite points equals this width, see the discussion in $\S$ 6 of the survey \cite{Cha-Gro}. Using relative differential geometry, Chakerian \cite{chakerian} studied this property already for normed planes and spaces, see also \cite{hug} and the discussion in $\S$ 2.7 of \cite{martini2}. We reprove it now for the planar case, using our methods prepared here.

\begin{prop} Let $\gamma:[0,c]\rightarrow X$ be a simple, closed and strictly convex curve of constant width $d \in \mathbb{R}$, which is at least of class $C^2$. Then the sum of the curvature radii of any pair of opposite point equals $d$. 
\end{prop}
\begin{proof} First, assume that $\varphi(u):[0,2\lambda(S)]\rightarrow X$ is a parametrization of the unit circle by twice the area of the sector, and re-parametrize $\gamma$ in such a way that we have
\begin{align*} \gamma'(u) = \rho(u)\varphi'(u).
\end{align*}
Setting things this way, we have that $k_c(u) = \rho(u)^{-1}$, i.e., $\rho(u)$ is precisely the curvature radius of $\gamma$ at $\gamma(u)$. Our first step is to prove that the line joining a pair of opposite points is left normal to the tangent direction of $\gamma$ at these points. The idea is to consider, for each point $\gamma(u)$, the point $q(u)$ of the opposite tangent line such that $\gamma(u) - q(u) \dashv_B \gamma'(u)$. We may write
\begin{align*} q(u) = \gamma(u + \lambda(S)) + g(u)\gamma'(u),
\end{align*}
for some function $g:[0,2\lambda(S)]\rightarrow \mathbb{R}$. We shall prove that $g$ equals zero. For this purpose, notice that since $\gamma$ has constant width we have $\gamma(u) - q(u) = d\varphi(u)$ for each $u\in [0,2\lambda(S)]$. Now, differentiating both equalities yields
\begin{align*} -\rho(u+\lambda(S))\varphi'(u) + g'(u)\gamma'(u) + g(u)\gamma''(u) = \gamma'(u) - d\varphi'(u).
\end{align*} 
The right hand side is a vector in the direction of $\varphi'(u)$. Therefore, the left hand side must be the same, and hence we have $g(u) = 0$.\\

In view of this argument, we may write $\gamma(u) - \gamma(u+\lambda(S)) = d\varphi(u)$ for any $u \in [0,2\lambda(S)]$. Now, we derivate to obtain
\begin{align*} d\varphi'(u) = \gamma'(u) - \gamma'(u+\lambda(S)) = [\rho(u) + \rho(u+\lambda(s))]\varphi'(u).
\end{align*} 
Therefore, we have indeed $\rho(u) + \rho(u+\lambda(S)) = d$ for any $u \in [0,2\lambda(S)]$, what we wanted to show. 

\end{proof}

\begin{coro} Let $\gamma:[0,c]\rightarrow X$ be a closed, simple and strictly convex curve of constant width $d \in \mathbb{R}$. Then, 
\begin{align*}l(\gamma) = d\frac{l(S)}{2}.
\end{align*}  
Moreover, if each pair of opposite points divides $\gamma$ into two portions of equal length, then $\gamma$ is a Minkowski circle. 
\end{coro}
\begin{proof} For the first claim, we make a simple calculation:
\begin{align*} l(\gamma) = \int_0^{2\lambda(S)}||\gamma'(u)||du = \int_0^{2\lambda(S)}\rho(u)||\varphi'(u)||du = \\ = \int_0^{\lambda(S)}\rho(u)||\varphi'(u)|| + \rho(u+\lambda(S))||\varphi'(u+\lambda(S))||du = \int_0^{\lambda(S)}d||\varphi'(u)||du = d\frac{l(S)}{2}.
\end{align*}
For the other claim we have to do a little more. We say that two arcs of $\gamma$ are opposite if their endpoints are respectively opposite. By suitably choosing the initial point and dividing successively, one can easily note that for any arc with length $d\frac{l(S)}{2^n}$, $n \geq 1$, the opposite arc has the same length. Now, assume that $\gamma$ is not a Minkowski circle. Then its circular curvature is not constant, and we can find some $u_0 \in [0,2\lambda(S)]$ for which $\rho(u_0) > d/2$. By continuity, we may choose a number $\epsilon > 0$ such that $\rho(u) > d/2$ in $[u_0,u_0+\varepsilon]$ and the length of $\gamma$ between $\gamma(u_0)$ and $\gamma(u_0+\varepsilon)$ equals $d\frac{l(S)}{2^n}$ for some $n \in \mathbb{N}$. It follows that $\rho(u) < d/2$ for every $u \in [u_0+\varepsilon,u_0+\lambda(S)+\varepsilon]$. But then we have
\begin{align*} \int_{u_0}^{u_0+\varepsilon}\rho(u)||\varphi'(u)||du > \int_{u_0}^{u_0+\varepsilon}\frac{d}{2}||\varphi'(u)||du > \int_{u_0}^{u_0+\varepsilon}\rho(u+\lambda(S))||\varphi'(u)||du = \\ = \int_{u_0+\lambda(S)}^{u_0+\lambda(S)+\varepsilon}\rho(u)||\varphi'(u)||du.
\end{align*}
Since the first and the latter are lengths of two opposite arcs, we have a contradiction. Then $\gamma$ must be a Minkowski circle.

\end{proof}

\begin{remark} The same holds for a curve with constant \emph{anti-width} (i.e., the width in the anti-norm) if we consider the \emph{anti-curvature radius} (i.e., the inverse of the normal curvature).
\end{remark}

\section{Evolutes, involutes and parallels}

Since we have the notion of osculating circle related to the circular curvature (and, of course, the notion of osculating anti-circle, associated with the normal curvature), it is natural to define and study the \emph{evolute} of a given \emph{regular curve} (i.e., a curve whose tangent vector field does not vanish). Among the papers in which these and other related concepts were studied we may refer to \cite{haddou}, \cite{biberstein}, \cite{craizer}, \cite{Cra-Tei-Ba}, \cite{Pet}, and \cite{tabachnikov}. For the discrete framework, see \cite{craizermartini} and \cite{Tei-Cra-Ba}. \\

Let $\gamma:[0,c]\rightarrow X$ be a smooth regular curve whose circular curvature does not vanish (hence the other curvatures also do not vanish), and assume also that $X$ is a smooth normed plane. Suppose that, for the sake of simplicity, $\gamma$ is parametrized by arc length (denoted, as usual, by $s$). Also, assume again that $t(s)$ is the function such that $\gamma'(s) = \frac{d\varphi}{dt}(t(s))$, where $t$ is an arc-length parameter in the unit circle. We define the \emph{evolute} of $\gamma$ to be the curve $\xi:[0,c]\rightarrow X$ given by
\begin{align} \label{evolute} \xi(s) = \gamma(s) - \rho(s)\varphi(t(s)),
\end{align}
where $\rho(s) = k_c(s)^{-1}$ is the curvature radius of $\gamma$ at $\gamma(s)$, as defined previously. In other words, the evolute of $\gamma$ is the locus of all its \emph{curvature centers} (i.e., the centers of its osculating circles). We shall give an example to illustrate the idea.

\begin{example}\label{lp} Consider $\mathbb{R}^2$ endowed with the usual $l_p$ norm 
\begin{align*}||(x,y)||_p = \left(|x|^p+|y|^p\right)^{1/p},
\end{align*}
where $1 < p < +\infty$. Let $q \in \mathbb{R}$ be such that $1/p + 1/q = 1$, and consider the parametrization $t\in[0,1] \mapsto (t^{1/p},(1-t)^{1/p})$ of the portion of the unit circle which lies in the first quadrant (the other portions are simple reflections of this). A tedious calculation shows that the evolute of the portion in the first quadrant of the $l_q$ circle $S_q:=\{(x,y)\in\mathbb{R}^2:|x|^q+|y|^q = 1\}$ is the curve whose coordinates are
\begin{align*} x(t) = t^{\frac{1}{q}}-\left(\frac{p}{q}\right)^2\frac{t^{\frac{p}{q}-\frac{1}{q}}}{(1-t)^{\frac{p}{q}-\frac{1}{q}+1}}\left(1+\left(\frac{t}{1-t}\right)^{\frac{p}{q}}\right)^{\frac{1}{p}-2} and\\
y(t) = (1-t)^{\frac{1}{q}} - \left(\frac{p}{q}\right)^2\frac{t^{\frac{p}{q}-1}}{(1-t)^{\frac{p}{q}}}\left(1+\left(\frac{t}{1-t}\right)^{\frac{p}{q}}\right)^{\frac{1}{p}-2},
\end{align*}
where $t \in (0,1)$, and we take limits to define the evolute in the extremities of the interval. To obtain the evolute of the entire $l_q$ circle, one just has to reflect the curve above through the coordinate axis and through the origin. Figure \ref{evolutelp} illustrates this evolute for the case $p = 3$. \\
\end{example}

\begin{figure}[h]
\centering
\includegraphics{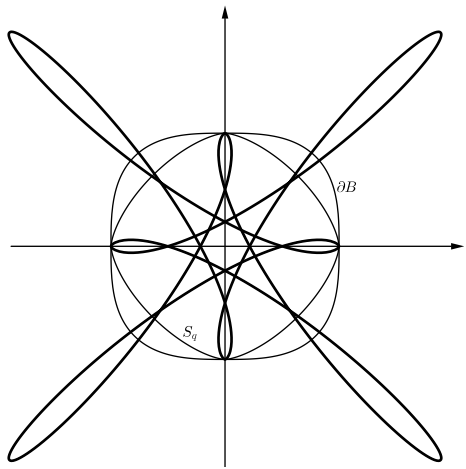}
\caption{The evolute of the $l_{\frac{3}{2}}$ circle in the $l_3$ norm.}
\label{evolutelp}
\end{figure}

Still considering the usual parametrization for the curve $\gamma$ and for the unit circle,  for each $s \in [0,c]$ we call the line $r \mapsto \gamma(s) + r\varphi(t(s))$ the \emph{left-normal line} of $\gamma$ at $\gamma(s)$. It is worth saying that Biberstein \cite{biberstein} and Craizer \cite{craizer} defined the evolute of a curve to be the envelope of its field of left-normal lines. It is easy to see that the definitions coincide (as it was made by both authors, indeed). One just has to differentiate (\ref{evolute}) to obtain
\begin{align}\label{evolutederivative} \xi'(s) = -\rho'(s)\varphi(t(s)),
\end{align}
and then $\xi'(s) \dashv_B \gamma'(s)$. In other words, the tangent line of the evolute at $\xi(s)$ is precisely the left-normal line of $\gamma$ at $\gamma(s)$. A third equivalent definition for the evolute of a curve was noticed by Craizer \cite{craizer} and comes from the viewpoint of singularity theory (see \cite{izumiya}).

\begin{prop} Define the \emph{squared distance function} of $\gamma$ to a fixed point $a \in X\setminus\{\gamma\}$ by $f(s) = ||\gamma(s) - a||^2$. Then the point $a$ lies in the left-normal line to $\gamma$ at $s_0$ if and only if $f'(s_0) = 0$ (i.e., $s_0$ is a singularity of type $A_{\geq 1}$ of $f$). Moreover, $a$ is the center of curvature of $\gamma$ at $s_0$ if and only if $f'(s_0) = f''(s_0) = 0$ (i.e., $s_0$ is a singularity of type $A_{\geq 2}$ of $f$). 
\end{prop} 
\begin{proof} Consider the function $D_a:X\rightarrow \mathbb{R}$ given by $D_a(x) = ||x-a||^2$. It is clear that its level curves are the Minkowski circles centered at $a$. This means that the kernel of the derivative of $D_a$ at a given point $b \in X\setminus\{a\}$ is given by the tangent direction at $b$ to the circle with center in $a$ and passing through $b$. Since $f(s) = D_a\circ \gamma$, it follows that $f'(s_0) = 0$ if and only if $\gamma'(s_0)$ points in the tangent direction at $\gamma(s_0)$ to the circle with center $a$ and passing through $\gamma(s_0)$. But this is the same as stating that $\gamma(s_0) - a \dashv_B \gamma'(s_0)$. \\

We come now to the other claim. We assume that $f'(s_0) = 0$ (and, consequently, that $\gamma(s_0) - a \dashv_B \gamma'(s_0)$) and prove that, in this case, $f''(s_0) = 0$ if and only if $a$ is the center of curvature of $\gamma$ at $\gamma(s_0)$. Since $a \notin \gamma$, at least locally we may write
\begin{align*} a = \gamma(s) - g(s)\varphi(\theta(s)), 
\end{align*}
where $g$ is a positive function, and $\theta$ is the same arc-length parameter as $t$. We just choose another letter to make the due distinction between them. Namely, while $t(s)$ is the function for which $\gamma'(s) = \frac{d\varphi}{dt}(t(s))$, the function $\theta(s)$ is such that $\varphi(\theta(s))$ points in the (oriented) direction of $\gamma(s) - a$. Differentiating, we have
\begin{align*} \gamma'(s) = g'(s)\varphi(\theta(s)) + g(s)\theta'(s)\frac{d\varphi}{d\theta}(\theta(s)).
\end{align*}
Hence, at $s_0$ we must have $g'(s_0) = 0$, $g(s_0)\theta'(s_0) = 1$ and $\theta(s_0) = t(s_0)$. Differentiating again and evaluating at $s_0$, we get
\begin{align*} \gamma''(s_0) = g''(s_0)\varphi(t(s_0)) + g(s_0)\theta''(s_0)\gamma'(s_0) + \theta'(s_0)\frac{d^2\varphi}{dt^2}(t(s_0)).
\end{align*}
Now remember that $\gamma''(s_0) = k_c(s_0)\frac{d^2\varphi}{dt^2}(t(s_0))$. It follows that $g''(s_0) = 0$ if and only if $\theta'(s_0) = k_c(s_0)$. On the other hand, the condition $g''(s_0) = 0$ is equivalent to $f''(s_0) = 0$, since we clearly have $f(s) = g(s)^2$. Also, $\theta'(s_0) = k_c(s_0)$ if and only if $g(s_0) = k_c(s_0)^{-1}$. Since $g(s) = ||\gamma(s) - a||$, we have the desired. 

\end{proof}

\begin{remark}\label{remarkevo} Notice that, in the last calculation, whenever the point $a$ is the center of curvature of $\gamma$ at $\gamma(s_0)$ we also have $\theta''(s_0) = 0$. \\
\end{remark}

We shall go a little further in this direction. A point $\gamma(s)$ of $\gamma$ is a \emph{vertex} if $k_c(s) \neq 0$ and $k_c'(s) = 0$. A vertex is called an \emph{ordinary vertex} if we also have $k_c''(s) = 0$. Analogously to the Euclidean case (see \cite{izumiya}), we have the following

\begin{prop} Let $f(s) = ||\gamma(s)-a||^2$ be the squared distance function of $\gamma$ with respect to $a$. Then we have $f'(s_0) = f''(s_0) = f'''(s_0) = 0$ (i.e., $f$ has a singularity of type $A_{\geq 3}$ at $s_0$) if and only if the point $a$ is the center of curvature of $\gamma$ at $\gamma(s_0)$ and $\gamma(s_0)$ is a vertex of $\gamma$. 
\end{prop}
\begin{proof} It suffices to assume that $a$ is the center of curvature at $\gamma(s_0)$ and to show that $f'''(s_0) = 0$ if and only if $k_c'(s_0) = 0$ and $k_c''(s_0) \neq 0$. We shall use the same notation as in the last proposition. Recall that in this case we have $g'(s_0) = g''(s_0) = 0$, $\theta(s_0) = t(s_0)$, $g(s_0)\theta'(s_0) = 1$, $\theta'(s_0) = k_c(s_0)$ and $\theta''(s_0) = 0$ (see Remark \ref{remarkevo}). Using this information and differentiating $\gamma''(s)$ in two ways at $s_0$, we get
\begin{align*} k_c'(s_0)\frac{d^2\varphi}{dt^2}(t(s_0)) + k_c(s_0)^2\frac{d^3\varphi}{dt^3}(t(s_0)) = \gamma'''(s_0) = \\ = g'''(s_0)\varphi(t(s_0)) + g(s_0)\theta'''(s_0)\frac{d\varphi}{dt}(t(s_0)) + \theta'(s_0)^2\frac{d^3\varphi}{dt^3}(t(s_0)),
\end{align*}
and hence
\begin{align*} k_c'(s_0)\frac{d^2\varphi}{dt^2}(t(s_0)) = g'''(s_0)\varphi(t(s_0)) + g(s_0)\theta'''(s_0)\frac{d\varphi}{dt}(t(s_0)).
\end{align*}
It follows that $g'''(s_0) = 0$ if and only if $k_c'(s_0) = 0$. It is clear that the condition $g'''(s_0) = 0$ is equivalent to $f'''(s_0) = 0$. 

\end{proof}

So far, evolutes in a normed plane behave analogously as in the Euclidean plane when we have the viewpoint of singularities of squared distance functions. But we do not need to stop here. In the Euclidean plane, a point $\gamma(s_0)$ is an ordinary vertex of $\gamma$ if and only if $f'(s_0) = f''(s_0) = f'''(s_0) = 0$, but $f^{(4)}(s_0) \neq 0$ (in other words, $s_0$ is a singularity of type $A_3$ of $f$). In a normed plane we have the following

\begin{prop} Let $\gamma(s_0)$ be a vertex of $\gamma$, with associated center of curvature in the point $a \in X\setminus\{\gamma\}$. Let also, as usual, $f(s) = ||\gamma(s) - a||^2$. Then $\gamma(s_0)$ is an ordinary vertex if and only if $f$ has a singularity of type $A_3$ at $s_0$.
\end{prop}
\begin{proof} If $\gamma(s_0)$ is a vertex of $\gamma$ and the point $a$ is the associated center of curvature, then we already have that $f$ has a singularity of type $A_{\geq 3}$ at $s_0$. In other words, we know that $f'(s_0) = f''(s_0) = f'''(s_0) = 0$. We shall prove that if the vertex is ordinary, then $f^{(4)}(s_0) \neq 0$. To do so, we first have to derivate $\gamma'''(s)$ and to use the informations on $g$, $k_c$, $\theta$ and their derivatives (at $s_0$) to obtain the equality
\begin{align*} g^{(4)}(s_0)\varphi(t(s_0)) + g(s_0)\theta^{(4)}(s_0)\frac{d\varphi}{dt}(t(s_0)) = k_c''(s_0)\frac{d^2\varphi}{dt^2}(t(s_0)).
\end{align*} 
Now, since any pair of the vectors above is linearly independent, it follows that $k_c''(s_0) \neq 0$ if and only if $g^{(4)}(s_0) \neq 0$. Since $g'(s_0) = g''(s_0) = g'''(s_0) = 0$, we have that $g^{(4)}(s_0) \neq 0$ if and only if $f^{(4)}(s_0) \neq 0$. 

\end{proof}

From (\ref{evolutederivative}) we have that the evolute of a regular curve is regular, except at points associated to the vertices of $\gamma$. We shall take a more careful look at the singularities of the evolute (i.e., the points where $\xi'(s) = 0$). A curve $\sigma:J\rightarrow X$ has an \emph{ordinary cusp} at $t_0 \in J$ if $\sigma'(t_0) = 0$, and $\sigma''(t_0)$ and $\sigma'''(t_0)$ are linearly independent. In the Euclidean plane, the evolute has ordinary cusps at points corresponding to ordinary vertices of $\gamma$, and this is true also in a normed plane. The proof consists of derivating $\xi'(s)$ twice and showing that if $s_0$ is an ordinary vertex of $\gamma$. Then the vector $\xi''(s_0)$ must have a component in the direction of $\varphi(t(s_0))$, and the vector $\xi'''(s_0)$ necessarily has a component in the direction of $\frac{d\varphi}{dt}(t(s_0))$. \\

We would like to regard cusps of the evolute from a geometric point of view. To do so, let $\gamma(s_0)$ be an ordinary vertex of $\gamma$, and assume that $s_0 \in (0,c)$. Then we know that $\xi'(s_0) = 0$. Since $\xi$ is the envelope of the left-normal lines of $\gamma$, it follows that both limits
\begin{align*} \lim_{s\rightarrow s_0^{\pm}}\frac{\xi'(s)}{||\xi'(s)||} 
\end{align*} 
exist and point in the direction of $\varphi(t(s_0))$. On the other hand, notice that 
\begin{align*}[\xi'(s),\varphi(t(s_0))] = -\rho'(s)\left[\frac{d\varphi}{dt}(t(s)),\varphi(t(s_0))\right],
\end{align*}
and since $\gamma(s_0)$ is an ordinary vertex, it follows that $\rho'(s)$ changes its sign at $s_0$. Therefore, $[\xi'(s),\varphi(t(s_0))]$ changes its sign at $s_0$, and this yields
\begin{align*} \lim_{s\rightarrow s_0^+}\frac{\xi'(s)}{||\xi'(s)||} = - \lim_{s\rightarrow s_0^-}\frac{\xi'(s)}{||\xi'(s)||}. 
\end{align*}
The geometric meaning of this is that the evolute \emph{changes its orientation} in an ordinary cusp (see Figure \ref{cusp}). This can be understood in view of (\ref{evolutederivative}), since the sign of $\rho'(s)$ changes when we pass through an ordinary cusp. As a consequence, we have that here, as in the Euclidean case, the \emph{signed} length of the evolute of a closed curve with positive curvature is zero. Indeed, 
\begin{align*} \int_0^c-\rho'(s)ds = -\rho(c) + \rho(0) = 0.
\end{align*}
This means that the two sums of the lengths of the respective portions intercalated by the cusps are equal. 

\begin{figure}[h]
\centering
\includegraphics{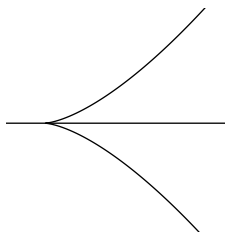}
\caption{An ordinary cusp.}
\label{cusp}
\end{figure}

A curve $\eta:[0,c]\rightarrow X$ is said to be an \emph{involute} of a curve $\gamma:[0,c]\rightarrow X$ if $\gamma$ is the evolute of $\eta$. We follow \cite{craizer} to find the involutes of a curve in a normed plane. Assume that $\gamma(s)$ is a curve parametrized by arc length and, to avoid any confusion, let $u$ be an arc-length parameter in the unit circle. Let $u(s):[0,c]\rightarrow [0,l(S)]$ be such that $\gamma'(s) = \varphi(u(s))$ for each $s \in [0,c]$. We claim that the curve $\eta:[0,c]\rightarrow X$ given by
\begin{align}\label{involute} \eta(s) = \gamma(s) + (c-s)\varphi(u(s)),
\end{align}
where $c \in \mathbb{R}$ is a constant, is an involute of $\gamma$. To check this, notice that
\begin{align*} \eta'(s) = \gamma'(s) - \varphi(u(s))+(c-s)\frac{d}{ds}\varphi(u(s)) = (c-s)\frac{d}{ds}\varphi(u(s)).
\end{align*}
Hence, $(c-s)$ is the curvature radius of $\eta$ at $\eta(s)$. Moreover, the left-normal direction to $\eta$ at $\eta(s)$ is $\varphi(u(s))$. Now, by the definition we have that the evolute $\xi_{\eta}$ of $\eta$ is given by
\begin{align*} \xi_{\eta}(s) = \eta(s) - (c-s)\varphi(u(s)) = \gamma(s).
\end{align*}
Actually, this construction yields every involute of $\gamma$. This is explained in the next proposition.
\begin{prop} Let $\gamma:[0,c]\rightarrow X$ be a curve parametrized by arc length, and let $u(s):[0,c]\rightarrow [0,l(S)]$ be as before. Then any involute of $\gamma$ must be of the form \emph{(\ref{involute})} for some constant $c \in \mathbb{R}$. 
\end{prop}
\begin{proof} Let $\eta$ be an involute of $\gamma$. Since $\gamma$ is the envelope of the left-normal lines of $\eta$, it follows that $\gamma'(s) \dashv_B \eta'(s)$ for any $s \in [0,c]$. Therefore, we may write $\eta'(s) = f(s)\frac{d}{ds}\varphi(u(s))$ for some function $f:[0,c]\rightarrow \mathbb{R}$. The evolute $\gamma$ of $\eta$ must then be given by $\gamma(s) = \eta(s) - f(s)\varphi(u(s))$. Differentiating, we have
\begin{align*} \varphi(u(s)) = \gamma'(s) = \eta'(s) - f'(s)\varphi(u(s)) - f(s)\frac{d}{ds}\varphi(u(s)) = -f'(s)\varphi(u(s)),
\end{align*}
and hence $f'(s)$ equals $-1$. It follows that $f(s) = c-s$ for some constant $c \in \mathbb{R}$.

\end{proof}

Let $\gamma$ be a regular curve. As usual, assume that $s$ is an arc-length parameter in $\gamma$, and that $\gamma'(s) = \frac{d\varphi}{dt}(t(s))$. The curve obtained by moving all the points of $\gamma$ at the same distance along each corresponding left-normal field is called a \emph{left parallel} of $\gamma$. Formally, a left parallel of $\gamma$ is a curve of the form
\begin{align*} \gamma_d(s) = \gamma(s) + d\varphi(t(s)),
\end{align*}
for some constant $d \in \mathbb{R}$. Notice that all the involutes of a given curve are left parallels of the same curve. Also, since $\gamma_d'(s) = (1+dk_c(s))\gamma'(s)$, we have that a left parallel is singular when $d = -k_c(s)^{-1}$. It follows that the evolute of $\gamma$ is precisely the curve whose points are the singularities of the left parallels of $\gamma$ (as in the Euclidean case, see \cite{izumiya}). \\

Of course, \emph{our little theory above can be completely analogized for anti-norms, just replacing the circular curvature by the normal curvature, and the left-normal line field by the right-normal line field.}

\section{Summary and related topics}

When studying the geometry of general normed planes, it is natural to investigate the differences to and the similarities with the Euclidean subcase. Even more, the concepts defined here are geometrically inspired or come directly from the analogous Euclidean concepts. We have shown how to obtain three (a priori) different curvature types by extending (to normed planes) three different ways of calculating or presenting the Euclidean curvature notion. In addition, we introduced a new kind of curvature motivated by analogous reasons and finally turning out as the Minkowski curvature of the anti-norm. This completes the whole framework also regarding duality concepts.\\

Therefore, natural questions may be posed, e.g. for geometric properties of each of these curvature types. We show, for example, that the curves of constant circular curvature are forming arcs of circles, and that the curves of constant normal curvature are forming arcs of anti-circles (the curves of constant Minkowski curvature are classified, too). This implies, in particular, the geometric fact that these are the only curves such that all the (left and right, respectively) normal lines meet at a point. We show also that each curvature type is invariant under an isometry of the plane, and that continuous functions over closed intervals can always be regarded as the curvature function (of any type) of a certain curve, which is unique up to the initial data. \\

Searching properties of curvature types that can characterize Radon planes, or the Euclidean plane, we proved that circular and normal curvature types coincide if and only the plane is Radon, and that if the Minkowski and normal curvature types coincide, then the norm is Euclidean. Even more, we gave a proof for the fact (mentioned, but not proved, by Petty \cite{Pet}) that if the Minkowski curvature of a circle is constant, then the plane must be Euclidean. \\

Using the same technique that Petty \cite{Pet} used for the Minkowski curvature, we proved that the four vertex theorem holds for all curvature concepts, also the new one. Still studying the curvature types for closed curves, we could show, as a consequence of a stronger result, that a simple, closed and strictly convex curve must be contained within the region bounded by its largest osculating circle, and must contain its smallest osculating circle. The methods to show this were also used to prove that a curve of constant width divided into two portions of equal length by any pair of opposite points must be a circle. \\

When studying evolutes, involutes and parallels, we adopted the modern approach of characterizing them by singularities of squared distance functions. It is somehow surprising how the things work in a way very similar to the Euclidean one. We emphasize that these results are also a good starting point for studying differential geometry of curves and surfaces in normed spaces of higher dimensions. \\

We believe that studying the geometry of normed spaces from the viewpoint of differential geometry will not only raise interesting geometric questions referring to these spaces, but can also help to achieve a better understanding regarding the objects we are dealing with in the Euclidean case. Of course, when one is willing to study systematically some rich subject within a different context, the questions to be raised and further directions of research are very numerous. Among them, we may cite, for example, the converse of the four vertex theorem (proved for the Euclidean case in \cite{gluck}, see also \cite{deturck}), which seems to be hard to tackle in normed planes. Further on, it seems that no analogue to Schur's theorem (see \cite[p. 406]{manfredo}) is proved for normed planes. \\

As for the case of planes, one can find in the literature also curvature concepts for normed spaces of higher dimensions (see, e.g., \cite{biberstein}, \cite{Bus3} and \cite{Gug2}). The notions and approaches used there (and also in various further papers) are, in general, regarding methods and objectives different to those presented here. A systematic treatment of this more general subject is analogously missing, and it is clearly a topic of forthcoming work. Another direction of future research is the extension of results presented here (or existing for higher dimensions) to non-symmetric distances or gauges, see, e.g., \cite{Gug2} and \cite{Ja-Ma-Ri}.

\bibliography{bibliography.bib}

\end{document}